\documentclass[a4paper]{article}
\usepackage[all]{xy}
\usepackage{mathrsfs}
\usepackage{amssymb}
\usepackage{amsmath}
\usepackage{latexsym}
\usepackage{amsthm}
\usepackage{graphics}
\usepackage[english]{babel}
\usepackage{graphicx}
\usepackage{etoolbox}
\usepackage{mathtools}
\usepackage{xcolor}
\usepackage[toc,page]{appendix}

\usepackage{tocbibind}
\usepackage{stmaryrd}              
\usepackage{verbatim}
\usepackage{mdframed}
\usepackage[margin=3.4 cm]{geometry}
\usepackage{enumerate}
\usepackage{bm}
\usepackage[colorlinks={true},linkcolor={blue},citecolor={blue}, urlcolor={blue}]{hyperref}
\usepackage[noabbrev,nameinlink]{cleveref}
\usepackage{comment}
\usepackage[textsize=footnotesize]{todonotes}

\usepackage[utf8]{inputenc}
\usepackage[T1]{fontenc}
\usepackage{lmodern}
\usepackage{microtype}
\usepackage{pgf,tikz,tikz-cd}
\usetikzlibrary{arrows}
\usetikzlibrary{intersections}
\usetikzlibrary{fadings}
\usetikzlibrary{shapes,positioning,quotes,decorations.markings}
\usetikzlibrary{spy} 
\theoremstyle{plain}

\newcommand{\bH}{\mathbb{H}}
\newcommand{\cE}{\mathcal{E}}
\newcommand{\cL}{\mathcal{L}}
\newcommand{\bN}{\mathbb{N}}
\newcommand{\bQ}{\mathbb{Q}}
\newcommand{\bZ}{\mathbb{Z}}
\newcommand{\bR}{\mathbb{R}}
\newcommand{\bC}{\mathbb{C}}

\newcommand{\cO}{\mathcal{O}}

\newcommand{\bP}{\mathbb{P}}

\newcommand{\isom}{\cong}

\newtheorem{thm}{Theorem}[section]
\newtheorem{prop}[thm]{Proposition}

\newtheorem{lemma}[thm]{Lemma}
\theoremstyle{definition}

\newtheorem{exe}[thm]{Example}

\newtheorem{rem}[thm]{Remark}
\newtheorem*{thm*}{Theorem (\Cref{mainTheoremTrue})}
\newtheorem*{thm*2}{Theorem (\Cref{AbhyankarApproach})}
\newtheorem*{thm*3}{Theorem (\Cref{mainTheoremStrong})}
\newtheorem*{thm*4}{Theorem (\Cref{commutators})}
\newtheorem*{prop*}{Proposition}

\newtheoremstyle{named}%
    {}{}{\itshape}{}{\bfseries}{.}{.5em}{\thmnote{#3}}
\theoremstyle{named}

\usepackage{blindtext}

\title{Monodromy of elliptic logarithms: \\ some topological methods and effective results}
\author{Francesco Tropeano}
\date{}

\newcommand{\Addresses}{{
  \bigskip
  \bigskip
  \footnotesize
}
  
F.~Tropeano, \textsc{Universit\`a della Calabria, Italy}\par\nopagebreak
  \textit{E-mail address}: \texttt{francesco.tropeano@unical.it}
  
}

\begin{document}

\maketitle

\makeatletter
\@starttoc{toc}
\makeatother

\pagebreak
\begin{abstract}
We present some effective approaches in studying the relative monodromy group of elliptic logarithms with respect to periods of elliptic schemes. We provide explicit ways of constructing explicit loops which leave periods unchanged but along which logarithms have non-trivial variations. We also get some topological methods and effective results which allow to manage the ramification locus of sections. The paper was inspired by a theorem of Corvaja and Zannier which abstractly determine the relative monodromy group of non-torsion sections.
\end{abstract}

\section{Introduction}

Period functions, elliptic logarithms and the Betti map are very useful tools in the context of elliptic and abelian schemes, for instance in studying the torsion points of a section and many related questions in `unlikely intersections'. Often, a key point in using these tools is that they are only locally defined as holomorphic functions: some natural problems emerge when we ask about their analytic continuation. Generally, it turns out that it's impossible to globally define these functions, in the sense that they turn out to be multi-valued functions and to have a non-trivial monodromy action when we travel along closed paths. We will deal with some questions relative to the said monodromy action, in particular this paper is devoted to the study of the monodromy of elliptic logarithms.

\paragraph{General notations.}
We denote by $B$ a smooth affine curve and by $\pi:\cE \rightarrow B$ a complex elliptic scheme, where $\cE$ is a quasi-projective surface and $\pi:\cE \rightarrow B$ is a surjective morphism all of whose fibers $\pi^{-1}(b)$ are elliptic curves. We denote by $\cE_b$ the fiber $\pi^{-1}(b)$. We always suppose that there exists a zero-section $\sigma_0:B \rightarrow \cE$ which marks the origin in each fiber. Moreover, we always suppose that the family is not isotrivial and that on any compactification $\overline{\cE}\rightarrow \overline{B}$ each fiber over $\overline{B}-B$ is singular. When speaking about paths in $B$ we use the word `loop' as synonymous of closed path and we usually employ the word `path' referring to paths which are not loops.\\

A natural problem consists in finding out the minimal (unramified) cover of $B$ on which the relevant functions can be globally defined, which means studying their monodromy. In this context, Corvaja and Zannier in \cite{CZ2} focused on elliptic logarithms proving that the relative monodromy group of logarithm of a non-torsion section $\sigma$ with respect to periods of $\cE \rightarrow B$ is non-trivial (more precisely they proved it is isomorphic to the group $\bZ^2$). In other words, this is equivalent to the assertion that there exists a certain loop $\Gamma$ in $B$ which leaves unchanged periods but not logarithm via analytic continuation. Anyway, they used cohomological techniques and holomorphic sections to obtain the conclusion but their proof does not give a method for constructing such a loop explicitly. This paper was inspired by the work of Corvaja and Zannier with the initial aim of providing an explicit construction of the loop $\Gamma$. As a result, our investigation led to a different approach to the problem which provides a new viewpoint on the issue and an alternative proof of \Cref{mainTheorem} which has the advantage of being \emph{effective}. Further, our approach provides methods for managing the ramification locus of a section $\sigma:B\rightarrow \cE$ and consequently the property of being non-torsion.

We will prove that it is enough to put ourselves in the case where the elliptic scheme $\cE \rightarrow B$ is a pullback of the Legendre elliptic scheme $\cL \rightarrow \bP_{1}-\{0,1,\infty\}$, so that we have a finite map $p:B \rightarrow \bP_{1}-\{0,1,\infty\}$. An initial investigation about loops with trivial monodromy action on periods leads us to define some subgroups $D^{(n)}$ as in \Cref{D0} and following. This ends with a result which is interesting on its own, determining explicitly all the loops in $B$ with trivial monodromy action on periods:

\begin{thm*4}
	We have
	$$
	\ker{\rho_B}=p_*^{-1}\left(\langle \bigcup_{n \in \bN} D^{(n)} \rangle\right),
	$$
	where $\rho_B$ is the monodromy representation of the elliptic scheme.
\end{thm*4}

However, our main focus remains the study of the relative monodromy group of elliptic logarithms by providing explicit constructions. As a main result, we will construct a suitable subset $\{b_1, \ldots, b_{N}\}$ of the fiber $p^{-1}(s)$ of a fixed point $s \in \bP_{1}-\{0,1,\infty\}$ and prove the following result:

\begin{thm*}
	Let $\sigma:B \rightarrow \cE$ be a non-torsion section of the elliptic scheme. There exists a loop $\Gamma$ in $B$ which induces trivial monodromy action on periods but non-trivial monodromy action on the logarithm; actually, there exists a point $i \in \{1, \ldots, N\}$ such that $\Gamma$ is of the form
	$$
	\Gamma=\alpha_{1}\delta_{i}\alpha_{i}^{-1}\delta_{i}^{-1},
	$$
	where $\alpha_{1}$ is a loop whose homotopy class is one of the generators of the fundamental group of $B$ (with base point $b_{1}$), $\alpha_{i}$ is a loop based at $b_{i}$ which satisfies $p\circ \alpha_{i} = p\circ \alpha_{1}$, and $\delta_{i}$ is a path from $b_{1}$ to $b_{i}$ which leaves periods unchanged via analytic continuation. In particular, the relative monodromy group of logarithm of $\sigma$ with respect to periods of $\cE \rightarrow B$ is non-trivial, hence isomorphic to $\bZ^2$.
\end{thm*}

Afterwards, we propose a second approach to the problem which needs an additional hypothesis on the ramification locus of the section but it has the advantage to give a more immediate way of constructing the loop $\Gamma$:

\begin{thm*2}
	Let $\sigma:B \rightarrow \cE$ be a non-torsion algebraic section of $\cL \rightarrow S$ and suppose that the ramification locus $\mathcal{R}$ of $\sigma$ does not contain the whole set $\{0,1,\infty\}$. Then the loop $\Gamma$ constructed in \Cref{mainTheoremTrue} can be chosen of the type
	$$
	\Gamma:=\alpha_1\delta_i\alpha_i^{-1}\delta_i^{-1},
	$$
	where the homotopy class of $\alpha_1$ is one of the generators of $\pi_1(B,b)$ and $\delta_i \in p_*^{-1}(D^{(0)})$.
\end{thm*2}

The previous theorems are focused on proving with effectivity the non-triviality of the relative monodromy group, which further is isomorphic to $\bZ^{2}$. This last conclusion is obtained with purely algebraic argument in a first moment, losing a little effectivity. Anyway, in the last part of the paper we provide an effective version of this step. To this end, let $z \in \bZ$ be the least common multiple of ramification indexes of points of $B$ which lie over $0$ or $1$ via the map $p:B \rightarrow \bP_{1}-\{0,1,\infty\}$. Denote by $\zeta_{0}$ and $\zeta_{1}$ the two loops in $B$ with base point $b_{1}$ constructed in \Cref{ramificationLoops}. The following result will provide a stronger conclusion than before in the sense of effectivity, exhibiting explicit generators of a full-rank subgroup of $M_{\sigma}^{\textnormal{rel}}$:

\begin{thm*3}
	Let $\sigma:B \rightarrow \cE$ be a non-torsion section of the elliptic scheme and let $\Gamma$ be the loop constructed in \Cref{mainTheoremTrue}. There exists a loop $\Gamma'$ in $B$ depending on $\Gamma$ and $z$ such that the monodromy actions of $\Gamma$ and $\Gamma'$ generate a subgroup of $M_{\sigma}^{\textnormal{rel}}$ which is isomorphic to $\bZ^{2}$; actually, $\Gamma'$ is of the form
	$$
	\Gamma'=\zeta\Gamma\zeta^{-1}\Gamma^{-1},
	$$
	where $\zeta \in \{\zeta_{0},\zeta_{1}\}$; further, keeping the notations of \Cref{commutators} we have $\Gamma' \in p_{*}^{-1}\left(D^{(z+1)}\right)$. In particular, the relative monodromy group of logarithm of $\sigma$ with respect to periods of $\cE \rightarrow B$ is isomorphic to $\bZ^2$.
\end{thm*3}

Several difficulties arise in our approach. The first one comes out when constructing the paths $\delta_{i}$: obtaining paths connecting two fixed points of $B$ with trivial monodromy action induced on periods of the Legendre scheme turns out to be not obvious. We reach this goal by getting some control on the ramification locus of the section. In a second moment we propose an alternative way of controlling the ramification of sections by using Abhyankar's lemma and Belyi pairs.\footnote{When our paper was finished, we realized that Belyi pairs also appear in connection with other problems in number theory as pointed out for instance in \cite{Z3}, \cite{Z2} and \cite{Z1}. However, our approach which consists in `separating the branch locus of sections' is completely new to our knowledge.} The second approach has some limitations but allows to obtain a more immediate way of constructing the $\delta_{i}$'s. Then, a second difficulty arises when controlling variations of logarithm along $\alpha_{1}$ and $\alpha_{i}$. This problem is fixed by using Mordell-Weil theorem for function fields and Shioda's theorem. Developing our topological tools we also obtain a way of constructing all the loops which acts trivially by monodromy on periods.

The previous theorems about the relative monodromy of elliptic logarithms are of independent interest since they contain as particular cases some classical theorems as: Manin's theorem on the non-constancy of the Betti map in the particular case of elliptic surfaces (see $\cite{Man}$ or \cite{Ber}); Shioda's theorem on the vanishing of the rank of elliptic modular surfaces (\cite[Theorem 5.1]{Shio}); a theorem of Bertrand in \cite{B} (generalized by André \cite[Theorem 3]{A}) which we state as \Cref{BertrandThm}. In particular, our methods and new proofs allow to get a topological and explicit approach to some classical results in transcendence of logarithms and periods. This kind of results are very useful tools in some problems in `unlikely intersections', in fact a result of independence between logarithms and periods is a standard step when we are interested in confining torsion of subvarieties of abelian schemes (for instance see \cite{CMZ}, \cite{MZ3}, \cite{MZ2}, \cite{MZ1} or \cite{CTZ}, \cite{DT}). We mention the fact that the \Cref{mainTheorem} has been conjecturally generalized in \cite{Tro}; in the same paper the author proved the instances of fibered product of elliptic schemes.

\paragraph{Structure of the paper.}
For the sake of clarity, we briefly explain how the paper is organized.

In the first section we introduce the general context fixing some notations and preliminary results needed in the following sections.

In the second part of the paper we focus on the main theorem and prove it is invariant under some assumptions and operations. We also provide some effective results, determining explicitly all the loops where periods have no monodromy.

The third part regards the first effective approach to the main theorem, which ends providing an explicit loop $\Gamma$ with desired properties. The method we use is based on some `control of the branch locus of the section', which is reached by using Legendre elliptic scheme, suitable factorizations of morphisms and topological considerations. The rest of the proof consists in controlling variations of elliptic logarithm by Mordell-Weil theorem and Shioda's theorem, concluding with an accurate analysis of analytic continuations of the relevant functions.

In the fourth part we propose a second approach to the problem with the aim of giving a more immediate construction of the loop $\Gamma$ in some cases. We analyze the advantages and limitations of this way. This approach provides a different method of controlling the branch locus, which we call `separation of the branch locus of the section': it is based on the use of Abhyankar's Lemma (here as \Cref{Abhyankar}) or Belyi pairs. The rest of the proof essentially follows the idea of our previous effective proof.

The last part is devoted to some additional considerations about our methods. In particular, we leave some questions and we give the concrete example of Masser's section showing how our ideas can be applied. Eventually, we obtain some stronger (effective) conclusions: in fact we determine explicit generators of a full-rank subgroup of the relative monodromy group of a non-torsion section. The last example about Masser's section also tacitly hides how the initial idea of the paper was born.

\textbf{Acknowledgements.} This work was part of the author's PhD thesis. The author is grateful to Umberto Zannier, Pietro Corvaja and Paolo Antonio Oliverio for helpful discussions, for their kind attention, their useful advice and references. The author also thanks Paolo Dolce for useful opinions and suggestions.

\section{Elliptic schemes and monodromy problems}

In this section we introduce some notations and preliminary results needed in the following sections. At the same time, we give an (unpretentious) overview on some monodromy problems on elliptic schemes.

Let us keep all the notations introduced above and denote by $I$ the interval $[0,1]$ on the real line. Moreover, we introduce the following notation, which will be repeatedly used when speaking about monodromy: if $\alpha:I \rightarrow S$ is a path and $f:S \rightarrow \bC$ is a function defined in a neighbourhood of $\alpha(0)$ that admits an analytic continuation along $\alpha$, we denote by $c_\alpha(f)$ the analytic continuation of $f$ in $\alpha(1)$ along $\alpha$.

Now, fix an elliptic scheme $\cE \rightarrow B$. Any fiber $\cE_b$ is analytically isomorphic to a complex torus $\bC/\Lambda_b$; thus in correspondence with the elliptic scheme we have a family of lattices $\Lambda_b$ varying with $b \in B$. We can locally define a basis of the period lattice: on suitable open simply connected subsets $U \subset B$ in the complex topology which cover $B$, we can find holomorphic functions
$$\omega_{U,1}, \omega_{U,2}: U \rightarrow \bC$$
such that $\omega_{U,1}(b), \omega_{U,2}(b)$ is a basis for $\Lambda_b$ for each $b \in U$.

We want to be a bit more explicit on this point. To this end, define $S:=\bP_1-\{0,1,\infty\}$ and denote by $\cL\rightarrow S$ the Legendre elliptic scheme, i.e. the elliptic scheme whose fiber corresponding to $\lambda \in S$ is given by
$$
\cL_\lambda: y^2=x(x-1)(x-\lambda).
$$
A possible choice of a basis of periods for the elliptic curve $\cL_\lambda$ is obtained by integrating the invariant differential $y^{-1}dx$ of the curve along two loops generating the homology of the torus corresponding to $\cL_\lambda$. We saw that for instance in the domain $\max(|\lambda|, |1-\lambda|)<1$ one obtains these periods as series expansions of suitable hypergeometric functions (see \cite{H}):
\begin{equation}\label{powerSeries}
	\omega^\cL_1(\lambda)=\pi\sum_{n=0}^\infty \binom{-\frac{1}{2}}{n}^2\lambda^n, \qquad \omega^\cL_2(\lambda)=i\pi\sum_{n=0}^\infty \binom{-\frac{1}{2}}{n}^2(1-\lambda)^n.
\end{equation}
The functions $\omega^\cL_1, \omega^\cL_2$ representing the periods in the above region, may be analytically continued along paths in $\bC-\{0,1\}$, obtaining functions $\omega^\cL_{U,1}, \omega^\cL_{U,2}$ which give locally a basis for the lattices $\Lambda_\lambda$ made up of holomorphic functions on open simply connected sets $U$ of an open covering of $\bC-\{0,1\}$.

We want to emphasize the fact that the Legendre elliptic scheme is more than a simple example. In fact, up to a finite base change, an (non-isotrivial) elliptic scheme $\cE \rightarrow B$ can be obtained as pullback of the Legendre scheme. Thus, in many situations if we prove a result on the Legendre elliptic scheme, we can say something about the starting elliptic scheme $\cE \rightarrow B$. To be more precise about this construction, we briefly explain that in the next remark:

\begin{rem}\label{twoTors}
	Let $\cE \rightarrow B$ be an elliptic scheme as above and suppose that the generic fiber of the family is given by a Weierstrass equation whose affine form is
	$$y^2=f(x),$$
	where the coefficients of $f$ are rational functions on $B$. Let us perform a finite cover $B'\rightarrow B$ of the base $B$, where $B'$ is chosen so large so that the two-torsion points of the generic fiber are well-defined functions on $B'$. We obtain an elliptic scheme $\cE'\rightarrow B'$ by pullback via $B'\rightarrow B$:
	$$\begin{tikzcd}
		\cE' \arrow{r} \arrow{d} & \cE \arrow{d}\\
		B' \arrow{r} & B.
	\end{tikzcd}$$
	Now, the $2$-torsion of the generic fiber is rational, so we can reduce the Weierstrass equation in Legendre form. This gives a morphism $B' \rightarrow S$ such that $\cE'$ is obtained as pullback of $\cL$ via the morphism $B' \rightarrow S$.
\end{rem}

Thanks to \Cref{twoTors} we can focus on elliptic schemes $\cE \rightarrow B$ endowed with a finite map $B\rightarrow S$ such that $\cE$ is obtained as pullback of the Legendre scheme via $B\rightarrow S$, i.e. we have an elliptic scheme $\cE \rightarrow B$ with a diagram
\begin{equation}\label{pullback}
	\begin{tikzcd}
		\cE \arrow{d} \arrow{r} & \cL \arrow{d}\\
		B \arrow{r}{p} & S,
	\end{tikzcd}
\end{equation}
where $p:B \rightarrow S$ is in general a ramified morphism.

We are able to construct explicitly periods of any elliptic scheme. We fix a base point $b \in B$, not a ramification point for $p:B \rightarrow S$, and let $s = p(b) \in S$. In a connected and simply connected neighborhood $U$ of $b$ in $B$, we can holomorphically define a basis $\omega_1, \omega_2$ of the period lattices by the equation
\begin{equation}\label{periodFunctions}
	\omega_i = \omega^\cL_i\circ p.
\end{equation}
Locally on suitable open subsets $U \subset B$ (in the complex topology), this gives a basis for $\Lambda_b$ made up of holomorphic functions $\omega_{U,1}, \omega_{U,2}:U \rightarrow \bC$.

\subsection{Monodromy of periods}

Combining \Cref{powerSeries} and \Cref{periodFunctions} we obtain locally period functions of any elliptic scheme $\cE\rightarrow B$. These functions may be analytically continued through the whole of $B$, but it's impossible to globally define them: they turn out to be multi-valued functions, i.e. they have quite nontrivial monodromy when traveling along closed paths. 

First of all, note that the monodromy action corresponds to the action of the fundamental group $G:=\pi_1(B,b)$ on the first homology group $H_1(\cE_b,\bZ)$: thus, we get a monodromy representation
\begin{equation}\label{monRepr}
	\rho=\rho_B:G \rightarrow \textnormal{Aut}(H_1(\cE_b,\bZ))\cong\textnormal{GL}_2(\bZ).
\end{equation}
Note that the period lattice as a whole does depend only on $b$, namely there is no monodromy (as a whole): in fact, the lattice is uniquely determined by the Weierstrass equation of the elliptic curve; also, the orientation of the tangent space coming from the complex structure provides an orientation of the set of bases for these lattices. To be more precise, since monodromy preserves the intersection form then any action of monodromy corresponds to an element of $\textrm{SL}_2(\bZ)$: therefore the above representation writes down as $\rho:G\rightarrow \textnormal{SL}_2(\bZ)$.

We define the monodromy group of $\cE\rightarrow B$ at $b$ as $\textnormal{M}(\cE,b)=\rho(G)$. Since $B$ is path connected, all the groups $\textrm{Mon}(\cE,b)$ are conjugate when we vary the base point $b \in B$: in other words, the subgroups $\textrm{Mon}(\cE)$ will be defined up to an inner automorphism of the group $\textrm{SL}_2(\bZ)$. Thus, fix once and for all a base point $b_0 \in B$ and denote the group $\textrm{Mon}(\cE,b_0)$ by $\textrm{Mon}(\cE)$, without writing any dependencies on the base point.

In view of \Cref{periodFunctions}, in order to obtain information about $\textrm{Mon}(\cE)$ it is enough to study the analytic continuation of power series in \Cref{powerSeries}. Since $\pi_1(S)$ is freely generated by two loops from the basepoint around $0,1$, define $\mathfrak{a}_0$ as a small circle centred at $0$ and $\mathfrak{a}_1$ as a small circle centred at $1$. The generators of $\pi_1(S)$ can be chosen as the homotopy classes $g_0, g_1$ of $\mathfrak{a}_0, \mathfrak{a}_1$, respectively. In \cite{JS} and \cite{CZ1}, explicit calculations about series expansions of periods are performed: in particular, in the second paper the authors explicitly exhibit the monodromy action on periods. They write down the following relations:
\begin{center}
	\begin{tabular}{lcl}
		$c_{\mathfrak{a}_0}(\omega^\cL_1) = \omega^\cL_1,$ & & $c_{\mathfrak{a}_0}(\omega^\cL_2) = \omega^\cL_2 + 2\omega^\cL_1;$\\
		$c_{\mathfrak{a}_1}(\omega^\cL_1) = \omega^\cL_1 + 2\omega^\cL_2,$ & & $c_{\mathfrak{a}_1}(\omega^\cL_2) = \omega^\cL_2.$
	\end{tabular}
\end{center}
This determines the full monodromy group $\textnormal{Mon}(\cL)$. In fact, given a period $\omega:=n\omega^\cL_1 + m\omega^\cL_2$, we obtain
\begin{equation*}
	\begin{aligned}
		c_{\mathfrak{a}_0}(\omega) &= n\omega^\cL_1 + m(\omega^\cL_2+2\omega^\cL_1) = (n+2m)\omega^\cL_1 + m\omega^\cL_2,\\
		c_{\mathfrak{a}_1}(\omega) &= n(\omega^\cL_1+2\omega^\cL_2) + m\omega^\cL_2 = n\omega^\cL_1 + (2n+m)\omega^\cL_2.
	\end{aligned}
\end{equation*}
In conclusion, the monodromy group $\textnormal{Mon}(\cL)$ is generated by the matrices
$$\left(\begin{matrix}1 & 0\\ 2 & 1\end{matrix}\right),
\qquad
\left(\begin{matrix}1 & 2\\ 0 & 1\end{matrix}\right),$$
which are known to generate freely the subgroup $\Gamma_2$ of $\Gamma(2) \subset \textrm{SL}_2(\bZ)$ consisting of matrices congruent to the identity modulo $2$ and with elements congruent to $1$ modulo $4$ on the diagonal. Hence, we obtain a (surjective) representation $\rho:\pi_1(S) \rightarrow \Gamma_2$ and the equality $\textnormal{Mon}(\cL)=\Gamma_2$. Precisely, we have an action of $\pi_1(S)$ on $\bZ^2$ (viewed as the coordinates of periods with respect to the basis $\omega_1^\cL, \omega_2^\cL$) as follows:
\begin{equation}\label{columnAction}
	g\cdot \left(\begin{matrix}
		n\\
		m
	\end{matrix}\right) =\rho(g)\left(\begin{matrix}
		n\\
		m
	\end{matrix}\right),
\end{equation}
where
$$\rho(g_0) := \left( \begin{matrix}
	1 & 2\\
	0 & 1
\end{matrix} \right),
\qquad 
\rho(g_1) := \left( \begin{matrix}
	1 & 0\\
	2 & 1
\end{matrix} \right).$$

Now consider a general elliptic scheme $\cE\rightarrow B$. By \Cref{twoTors}, we can assume to have a finite map $B \rightarrow S$ which realizes $\cE\rightarrow B$ as pullback of the Legendre elliptic scheme with periods given by \Cref{periodFunctions}. Hence, the monodromy group $\textrm{Mon}(\cE)$ is a finite-index subgroup of $\textrm{SL}_2(\bZ)$. Moreover, since $\Gamma_2$ is Zariski-dense in $\textrm{SL}_2(\bZ)$ and $\textrm{SL}_2(\bZ)$ acts irreducibly on the lattice of periods then the monodromy group $\textrm{Mon}(\cE)$ is Zariski-dense in $\textrm{SL}_2(\bZ)$ and acts irreducibly on the lattice of periods.

\begin{rem}\label{wellDefPeriods}
	The previous considerations yield some conclusion on the good definition of periods on the base $B$. Let $\omega_1, \omega_2$ be a basis for the period lattice, defined locally as complex functions on $U \subset B$, where $U$ is a simply connected open set of $B$. Since the group $\textrm{Mon}(\cE)$ is non-trivial, the functions $\omega_1, \omega_2$ cannot be both defined continuously on the whole of $B$. Moreover, since the action of $\textrm{Mon}(\cE)$ on the lattice of periods is irreducible, neither $\omega_1$ nor $\omega_2$ nor any single non-zero element of the period lattice can be defined on the whole of $B$.
\end{rem} 

\subsubsection{Modular elliptic schemes}

At a certain point, we shall be interested in a particular type of elliptic schemes: the elliptic schemes obtained from the Legendre scheme $\cL \rightarrow S$ by pullback
under an \emph{unramified} cover $p:B \rightarrow S$. These elliptic schemes (or, to be
more precise, their compactifications) are named modular elliptic surfaces in Shioda's paper \cite{Shio}. We call them \emph{modular elliptic schemes} to underline the fact that we only have smooth fibers.

Let $\cE \rightarrow B$ a modular elliptic scheme. Obviously, this is a particular case of an elliptic scheme which satisfies \Cref{pullback}: in addition, we have that $p$ is unramified. Thus, we can define periods as in \Cref{periodFunctions} and consider the associated monodromy action $\rho$ as in \Cref{monRepr}. In this case, \cite[Theorem 5]{Asa} ensures the representation $\rho$ to be faithful: we obtain an embedding $\rho:G \hookrightarrow \Gamma_2$.

Generally, if $\cE \rightarrow B$ is not a modular elliptic scheme, the monodromy action on periods is not faithful. The kernel of the action coincides with the kernel of the induced map $p_*:G \rightarrow \Gamma_2$ between fundamental groups.

In this context, we will be interested in \cite[Theorem 5.1]{Shio} which we will use as stated in \cite[Theorem 2.5]{CZ2}. We write down this last statement, since it will be an important tool in our approaches:
\begin{thm}\label{ShiodaThm}\textbf{\textit{(Shioda's Theorem)}}
	Every (rational) section $\sigma:B\rightarrow \cE$ of a modular elliptic scheme is torsion.
\end{thm}

\subsection{Monodromy group of a section}

Now, let us consider a section $\sigma:B\rightarrow \cE$ of an elliptic scheme. Over every point $b \in B$, we have an elliptic exponential map $\textrm{Lie}(\cE_b) \rightarrow \cE_b$, whose kernel is the period lattice. The family of Lie algebras $\textrm{Lie}(\cE)\rightarrow B$ defines a line bundle over $B$ (it is actually a complex Lie algebra bundle), which is trivial in the holomorphic sense and it can then be analytically identified with the product $B\times \bC$. By using the fiberwise exponential maps one can define a global map
$$
\exp: B\times \bC \rightarrow \cE.
$$
Observe that by restricting to $U \times \bC$, we obtain the covering map
$$U\times \bC\rightarrow \cE_{|U},$$
where we are denoting $\pi^{-1}(U)$ by $\cE_{|U}$. Given a simply connected open set $U\subset B$ where periods are well-defined, we define an \emph{elliptic logarithm of $\sigma$ in $U$} as a lifting of $\sigma_{|U}$ to $U\times \bC$: precisely, we define the logarithm as a function $\log_\sigma:U\rightarrow \bC$ which makes commutative the diagram
$$\begin{tikzcd}
	& U\times \bC \arrow{d}\\
	B \supset U \arrow[swap]{r}{\sigma_{|U}} \arrow[dashed]{ur}{(\textnormal{id},\log_\sigma)} & \pi^{-1}(U) \subset \cE.
\end{tikzcd}$$

In this context, we simply mention the fact that considering the real coordinates of elliptic logarithms with respect to the basis of periods we obtain the so called Betti map $\beta_\sigma:U \rightarrow \bR^2$: this is a real-analytic map and is very useful in many problems in `unlikely intersections'.

Similarly to period functions, once we have locally defined the logarithm of a section we can think about its analytic continuation through the whole of $B$. As a first example, let us consider the zero-section $\sigma_0$ which associates to each $b \in B$ the origin $O_b$ of the corresponding fiber $\cE_b$. A logarithm of $\sigma_0$ is given by the zero function
$$\log_{\sigma_0}:B \rightarrow \bC, \qquad \log_{\sigma_0}(b)=0 \textrm{ for each } b \in B.$$
Thus, in this case we can find a well-defined logarithm on the whole of $B$, in fact it has no monodromy. For rational sections, this is the only case in which such a global logarithm exists. For the sake of completeness, we briefly resume some approaches to a proof of this.

\begin{rem}\label{wellDefLog}
	Let's give a sketch of a proof of the following fact: a non-zero rational section $\sigma:B \rightarrow \cE$ cannot admit a well-defined logarithm on the whole of $B$. A way to prove it is by using Mordell-Weil theorem for the Legendre curve over the function field of $B(\bC)$ in a way that appeared also in \cite{Man}. If a non-zero rational section admits an elliptic logarithm which is well defined on the whole of $B(\bC)$, then we may divide it, and hence the section, by any prescribed positive integer and again we have maps well defined on $B(\bC)$. Thus the section would be infinitely divisible on $B(\bC)$ (since the submultiples of the sections would be algebraic and well defined on $B(\bC)$, hence rational on $B(\bC)$). But this violates the Mordell-Weil theorem for the Legendre curve over the function field of $B(\bC)$. In \Cref{wellDefLog2} we recall an alternative argument to prove this statement with different technique.
\end{rem}

In general, the fundamental group $G$ of $B$ acts by monodromy on logarithms of $\sigma$. Clearly, two determinations of the logarithm (defined over an open set $U \subset B$) differ by an integral combination of $\omega_1, \omega_2$. Thus, the action of $G$ on the determinations of the logarithm is as follows: fixed a determination of $\log_\sigma$ and given a loop $\alpha_g$ which represents the homotopy class $g \in G$ we have
\begin{equation}\label{actionLog}
	g\cdot\log_\sigma=c_{\alpha_g}(\log_\sigma)=\log_\sigma + u_g\omega_1 + v_g\omega_2,
\end{equation}
where $u_g, v_g \in \bZ$ and $c_{\alpha_g}$ denotes the analytic continuation along $\alpha_g$. Let us consider the map 
$$G \rightarrow \bZ^2, \qquad g \mapsto (u_g,v_g).$$

Let's continue denoting by $\omega_1, \omega_2$ a basis of periods and by $\log_\sigma$ a logarithm of $\sigma$. Choose $g,h \in G$ and denote by $\bar{g}:=\rho(g), \bar{h}:=\rho(h) \in \textrm{Mon}(\cE)$ their monodromy representations. By looking at the action of $gh$ and recalling \Cref{columnAction}, we have
\begin{align*}
	\log_\sigma \xrightarrow{h} \log_\sigma + (u_h, v_h)\left(\begin{matrix} \omega_1\\ \omega_2\end{matrix}\right) \xrightarrow{g} \log_\sigma + (u_g, v_g)\left(\begin{matrix} \omega_1\\ \omega_2\end{matrix}\right) + (u_h, v_h)\bar{g}^t\left(\begin{matrix} \omega_1\\ \omega_2\end{matrix}\right).
\end{align*}
In other words, we obtain
$$(u_{gh},v_{gh}) = (u_g, v_g) + (u_h, v_h)\bar{g}^t,$$
i.e. the map $g \mapsto (u_g,v_g)$ \emph{is a cocycle for the described action of $G$ on $\bZ^2$}.

The cocycle just defined describes the obstruction for a(n analytic) section to have a well-defined logarithm. In fact, a section which has a well-defined logarithm is characterized by the fact that the map $g\mapsto (u_g,v_g)$ is a coboundary for the above mentioned action, i.e. there exists a fixed vector $(u,v) \in \bZ^2$ such that $(u_g, v_g) = (u,v)(\bar{g}^t-I)$ for all $g \in G$. This statement is clarified in the following proposition:
\begin{prop}\label{coboundary}
	Let $\sigma:B\rightarrow \cE$ be an analytic section and $\log_\sigma$ a determination of its logarithm over an open set $U\subset B$. The section admits a well-defined logarithm on $B$ if and only if the associated cocycle $g \mapsto (u_g,v_g)$ is a coboundary.
\end{prop}

\begin{proof}
	Suppose that $\sigma$ admits a well-defined logarithm $\ell:B \rightarrow \bC$. Then the two determinations $\ell$ and $\log_\sigma$ over $U$ differ by a period $\omega:=n\omega_1+m\omega_2$, i.e.
	$$\log_\sigma=\ell + n\omega_1 + m\omega_2.$$
	For $g \in G$, we have
	\begin{align*}
		g\cdot\log_\sigma &= g\cdot (\ell + n\omega_1 + m\omega_2) = \ell + (n,m)\bar{g}^t\left(\begin{matrix} \omega_1 \\ \omega_2\end{matrix}\right)=\\
		&= \log_\sigma + \left[(n,m)\bar{g}^t-(n,m)\right] \left(\begin{matrix} \omega_1 \\ \omega_2\end{matrix}\right).
	\end{align*}
	Thus the corresponding cocycle is given by
	$$g\mapsto (n,m)\bar{g}^t-(n,m)$$
	for $g \in G$ and a fixed pair $(n,m) \in \bZ^2$, hence it is a coboundary.
	
	Viceversa, let us suppose to have a logarithm $\log_\sigma$ over $U$ and that there exists a fixed pair $(n,m) \in \bZ^2$ such that
	$$g\cdot\log_\sigma = \log_\sigma + \left[(n,m)\bar{g}^t-(n,m)\right] \left(\begin{matrix} \omega_1 \\ \omega_2\end{matrix}\right)$$
	for each $g \in G$. Let us define the function
	$$\ell:=\log_\sigma - n\omega_1 - m \omega_2,$$
	which is another determination of the logarithm of $\sigma$. Looking at the action of $G$ we obtain
	\begin{align*}
		g\cdot \ell &= g\cdot(\log_\sigma - n\omega_1 - m \omega_2) =\\
		&=\log_\sigma + \left[(n,m)\bar{g}^t-(n,m)\right] \left(\begin{matrix} \omega_1 \\ \omega_2\end{matrix}\right) - (n,m)\bar{g}^t \left(\begin{matrix} \omega_1 \\ \omega_2\end{matrix}\right)=\\
		&= \log_\sigma - n\omega_1 - m \omega_2 =\ell.
	\end{align*}
	Therefore, $\ell$ is a well-defined logarithm of $\sigma$ on the whole of $B$.\\
\end{proof}

The above argument leads to consider the cohomology group $H^1(G,\bZ^2)$, which is the method used in \cite{CZ2}.

Now, let us look at the simultaneous monodromy action of $G$ on periods and logarithm. By \Cref{columnAction} and \Cref{actionLog}, we can provide a new representation $$\theta_\sigma: G \rightarrow \textrm{SL}_3(\bZ),$$
where every matrix $\theta_\sigma(g)$ is of the form
\begin{equation}\label{simultaneousRepresentation}
	\theta_\sigma(g)=\left( \begin{matrix} \rho(g) & w_g \\ 0 & 1 \end{matrix} \right),
\end{equation}
where $w_g=(u_g,v_g)^t \in \bZ^2$. Note that the matrix $\rho(g)$ acts on the periods as specified in \Cref{columnAction}, and does not depend on $\sigma$. Moreover, the vector $w_g$ encodes the action of $g$ on determinations of the logarithm $\log_\sigma$ as in \Cref{actionLog}. Define the monodromy group of the section $\sigma$ as $M_\sigma:=\theta_\sigma(G)$.

\begin{rem}\label{wellDefLog2}
	As mentioned in \Cref{wellDefLog}, we recap an alternative argument to prove that a non-zero rational section $\sigma:B \rightarrow \cE$ cannot admit a well-defined logarithm on the whole of $B$. In this case, we use the just introduced representation $\theta_{\sigma}$ and the mentioned theorem of Bertrand, which we state here for the sake of completeness:
	\begin{thm}\label{BertrandThm}
	If the (rational) section $\sigma:B \rightarrow \cE$ is non-torsion, the kernel of the homomorphism $\theta_\sigma(G)^\textrm{Zar} \rightarrow \textrm{SL}_2$ is isomorphic to $\mathbb{G}_a^2$. In particular, the algebraic group $\theta_\sigma(G)^\textrm{Zar}$ has dimension five.
\end{thm}
	If a rational section $\sigma$ admits a logarithm, then after conjugation $\theta_\sigma(g)$ has $w_g=(0,0)^t$ for all $g \in G$. This means that $\theta_\sigma(G)^{\textrm{Zar}}$ has not dimension five, thus $\sigma$ is a torsion section by Bertrand's theorem. Then a logarithm of $\sigma$ has the form $\log_\sigma=q_1\omega_1 + q_2\omega_2$, where $q_1, q_2 \in \bQ$. We can see, by explicit computation, that such a function is well-defined on the whole of $B$ if and only if $q_1=q_2=0$, i.e. $\sigma$ is the zero-section.
\end{rem}

\subsection{Relative monodromy group of a section}

Let's continue considering an elliptic scheme $\cE \rightarrow B$ and a rational non-zero section $\sigma:B \rightarrow \cE$. By our previous considerations, neither a basis of periods nor a logarithm $\log_\sigma$ can be well-defined on the whole of $B$. Anyway, the relevant functions can be globally defined on the universal cover $\bH$ of $B$; note that $B$ is hyperbolic in view of our hypothesis that the scheme is not isotrivial. Studying the related monodromy problems corresponds to finding out the minimal (unramified) cover of $B$ on which both a basis of the periods and a logarithm of the section can be defined. With this in mind, we first call $B^* \rightarrow B$ the minimal (unramified) cover of $B$ on which a basis for the period lattice can be globally defined and we set $B_\sigma \rightarrow B^*$ to be the minimal cover of $B^*$ where one can define the logarithm of $\sigma$. The tower of covers is represented in the diagram:
\begin{equation}\label{monodromyDiagram}
	\bH \rightarrow B_\sigma \rightarrow B^* \rightarrow B.
\end{equation}
In particular, the group $\textnormal{Mon}(\cE)$ corresponds to the Galois group of the covering map $B^* \rightarrow B$, while the group $M_\sigma$ corresponds to the Galois group of the covering map $B_\sigma\rightarrow B$. Our interest is in studying the relative monodromy of the logarithm of a section with respect to the monodromy of periods, i.e. studying the covering map $B_\sigma \rightarrow B^*$. Topologically, this is the same as looking at the variation of logarithm via analytic continuation along loops on $B$ which leave periods unchanged. Moreover, in terms of \Cref{monRepr} and \Cref{simultaneousRepresentation} this corresponds to studying the group $M_\sigma^{\textnormal{rel}}:=\theta_\sigma(\ker{\rho})$, which we define as \emph{relative monodromy group of $\sigma$}.

\begin{rem}
	Note that the Zariski-closures of the discrete groups $\textnormal{Mon}(\cE), M_\sigma$ and $M_\sigma^{\textnormal{rel}}$ are the differential Galois groups of some Picard-Vessiot extensions of $\bC(B)$ obtained with $\omega_1, \omega_2, \log_\sigma$ and their derivatives (see \cite{CH} for further details about differential Galois theory and Picard-Vessiot extensions). In these terms, the differential Galois group $\overline{M_\sigma^\textnormal{rel}}$ was just determined in \cite{B}; anyway, this result says nothing on the relative monodromy of the logarithm over $B^*$ because it involves the Zariski closure of $M_{\sigma}$, and this information is not as strong as we want. In fact, in \cite{CZ1} and \cite{CZ2} the authors pointed out that it may happen that the Zariski closure of a group contains quite limited information on the group itself since it may be larger than expected in comparison with the group itself. Thus, in order to obtain information on $M_\sigma^\textnormal{rel}$, we have to introduce considerations of different nature with respect to Bertrand's theorem.
\end{rem}

In \cite[Theorem 2.1]{CZ2} the authors proved a result which we can state as follows:

\begin{thm}\label{mainTheorem}
	If $\sigma: B \rightarrow \cE$ is a non-torsion (rational) section, then $M_\sigma^{\textnormal{rel}}\cong \bZ^2$.
\end{thm}

Note that $M_\sigma^{\textnormal{rel}}$ is clearly a subgroup of $\bZ^2$; the previous theorem asserts that this subgroup is as large as possible (up to isomorphisms). Let us illustrate what happens in the case not covered by \Cref{mainTheorem}, i.e. when the section is torsion; moreover the example aims at introducing our topological approach more concretely.

\begin{exe}\label{exaTorsSec}
	Let us consider a torsion section $\sigma:B \rightarrow \cE$. By the properties of Betti map, any logarithm of $\sigma$ is a rational constant combination of periods; i.e.
	$$
	\log_\sigma = q_1 \omega_1 + q_2 \omega_2,
	$$
	where $q_1, q_2 \in \bQ$. Therefore, a loop which leaves unchanged periods via analytic continuation, leaves also unchanged the logarithm of such a section. In other words, the cover $B_\sigma \rightarrow B^*$ is trivial in this case and then $M_\sigma^\textnormal{rel}\cong \{0\}$.
\end{exe}

\section{Preliminary results and first thoughts on effectivity}

\subsection{Invariance results}

Let's start by proving some preliminary results about \Cref{mainTheorem} and the group $M_\sigma^\textnormal{rel}$.

\begin{lemma}\label{baseChange}
	Theorem $\ref{mainTheorem}$ is invariant by finite base change: in other words, if $\varphi: \widetilde{B} \rightarrow B$ is a finite morphism and $\widetilde{\cE}$ is the pullback of $\cE\rightarrow B$ via $\varphi$, then the theorem for $\widetilde{\cE} \rightarrow \widetilde{B}$ implies the theorem for $\cE \rightarrow B$.
\end{lemma}

\begin{proof}
	Let us denote by $\omega_1, \omega_2$ a basis of periods of $\cE \rightarrow B$. We can define $\widetilde{\omega_1}, \widetilde{\omega_2}$, a basis of periods of $\widetilde{\cE} \rightarrow \widetilde{B}$, by the equations
	\begin{equation}\label{eqn3}
		\widetilde{\omega_1}=\omega_1 \circ \varphi, \qquad \widetilde{\omega_2}=\omega_2 \circ \varphi.
	\end{equation}
	Fix two base points $\widetilde{b} \in \widetilde{B}$ and $b \in B$ such that $\varphi(\widetilde{b})=b$; we omit to explicitly write them in the fundamental group notation. If we denote by $\bar{G}, \widetilde{G}$ the monodromy groups of periods of $\cE \rightarrow B$, $\widetilde{\cE} \rightarrow \widetilde{B}$ respectively, we obtain the associated monodromy representations:
	$$\rho:\pi_1(B) \rightarrow \bar{G}, \qquad \widetilde{\rho}:\pi_1(\widetilde{B}) \rightarrow \widetilde{G}.$$
	By $(\ref{eqn3})$, we have $\widetilde{\rho}(g)=\rho(\varphi_*(g))$ for each $g \in \pi_1(\widetilde{B})$ (here $\varphi_*:\pi_1(\widetilde{B})\rightarrow \pi_1(B)$ is the induced map between fundamental groups). In particular, we obtain
	\begin{equation}\label{eqn3_bis}
		\varphi_*(\ker{\widetilde{\rho}})=\ker{\rho} \cap \varphi_*(\pi_1(\widetilde{B})).
	\end{equation}
	Let $\sigma:B \rightarrow \cE$ be a non-torsion section of $\cE \rightarrow B$. Since $\widetilde{\cE}\rightarrow \widetilde{B}$ is obtained as pullback of $\cE \rightarrow B$, then the elliptic curves $\widetilde{\cE}_{\widetilde{b}}$ and $\cE_b$ are canonically identified; therefore the pullback $\varphi^*(\sigma):=\sigma\circ \varphi:\widetilde{B} \rightarrow \widetilde{\cE}$ is a non-torsion section of $\widetilde{\cE} \rightarrow \widetilde{B}$. We have the associated monodromy representations
	$$\theta_\sigma:\pi_1(B) \rightarrow \textrm{SL}_3(\bZ), \qquad \theta_{\varphi^*(\sigma)}:\pi_1(\widetilde{B}) \rightarrow \textrm{SL}_3(\bZ).$$
		
	Using again the fact that the elliptic curves $\widetilde{\cE}_{\widetilde{b}}$ and $\cE_b$ are canonically identified, we obtain that a determination of the logarithm of $\varphi^*(\sigma)$ over $\widetilde{b}$ can be defined by the equation
	\begin{equation}\label{eqn4}
		\log_{\varphi^*(\sigma)}(\widetilde{b}):=\log_\sigma(b).
	\end{equation}
	Hence, we have $\theta_{\varphi^*(\sigma)}(g)=\theta_\sigma(\varphi_*(g))$ for each $g \in \pi_1(\widetilde{B})$.
	
	Since we are supposing that \Cref{mainTheorem} holds for $\widetilde{\cE}\rightarrow \widetilde{B}$, then we have
	$$
	M_{\varphi^*(\sigma)}^{\textnormal{rel}}=\theta_{\varphi^*(\sigma)}(\ker{\widetilde{\rho}}) \isom \bZ^2.$$
	By \Cref{eqn3_bis} and \Cref{eqn4}, this means
	$$
	\theta_{\sigma}(\ker{\rho} \cap \varphi_*(\pi_1(\widetilde{B}))) = \theta_\sigma(\varphi_*(\ker{\widetilde{\rho}})) = M_{\varphi^*(\sigma)}^{\textnormal{rel}} \isom \bZ^2.$$
	Since $\ker{\rho} \cap \varphi_*(\pi_1(\widetilde{B})) \subset \ker{\rho}$, we get $M_{\varphi^*(\sigma)}^{\textnormal{rel}} \subseteq M_{\sigma}^{\textnormal{rel}}$. This implies
	$$
	M_\sigma^{\textnormal{rel}} \isom \bZ^2,
	$$
	which is equivalent to saying that Theorem $\ref{mainTheorem}$ holds for $\cE \rightarrow B$.\\
\end{proof}

Let us consider a non-torsion section $\sigma:B \rightarrow \mathcal{E}$. Note that the representation $\theta_\sigma$ is defined in terms of \Cref{actionLog}, thus it depends on the branch of logarithm we fix. Here, we want to prove that $M_\sigma^\textnormal{rel}$ is independent of the choice of branch of $\log_\sigma$ and remains unchanged under some operations on the section.

\begin{prop}\label{relMonGrp}
	Let $\sigma: B \rightarrow \cE$ be a non-torsion section. Then
	\begin{itemize}
		\item[(i)] the group $M_\sigma^\textnormal{rel}$ does not depend on the choice of branch of $\log_\sigma$;
		
		\item[(ii)] the groups $M_\sigma^\textnormal{rel}$ and $M_{n\sigma}^\textnormal{rel}$ are isomorphic.
	\end{itemize}
\end{prop}

\begin{proof}
	$(i)$: Choose two branches of $\log_\sigma$, say $\ell_\sigma^1$ and $\ell_\sigma^2$ and denote by $M_{\sigma,1}^\textnormal{rel}$ and $M_{\sigma,2}^\textnormal{rel}$ the corresponding relative monodromy groups. Since the two branches $\ell_\sigma^1$ and $\ell_\sigma^2$ differs by a period, we get the thesis. In fact, since any loop $\alpha$ in $B$ whose homotopy class lies in $\ker{\rho}$ leaves periods unchanged, then $\ell_\sigma^1$ and $\ell_\sigma^2$ have the same variation by analytic continuation along $\alpha$. Thus we get $M_{\sigma,1}^\textnormal{rel}=M_{\sigma,2}^\textnormal{rel}$.
	
	$(ii)$: Fixed a branch $\ell_\sigma$ of $\log_\sigma$, a determination $\ell_{n\sigma}$ of logarithm of $n\sigma$ can be defined by the equation
	$$
	\ell_{n\sigma}=n\ell_\sigma.
	$$
	For any loop $\alpha$ we have the corresponding variations:
	$$
	c_\alpha(\ell_{n\sigma})=\ell_{n\sigma}+\omega_\alpha^{n\sigma}, \qquad  c_\alpha(\ell_\sigma) = \ell_\sigma+\omega_\alpha^\sigma,
	$$
	where $\omega_\alpha^{n\sigma}=n\omega_\alpha^\sigma$. Thus $M_\sigma^\textnormal{rel}$ and $M_{n\sigma}^\textnormal{rel}$ are isomorphic.\\
\end{proof}

\subsection{Some `effective' considerations and results}

Let $\cE \rightarrow B$ an elliptic scheme which satisfies \Cref{pullback} and look at the relative monodromy group $M_\sigma^{\textnormal{rel}}$ of a section $\sigma$. When the morphism $p:B \rightarrow S$ is unramified we are dealing with a modular elliptic scheme; since the representation $\rho$ is faithful in this case, we obtain $\ker{\rho}=\{0\}$.\footnote{As a consequence of faithfulness of $\rho$ and \Cref{mainTheorem}, we find again \Cref{ShiodaThm}. However, note that this is not a new proof of Shioda's theorem, since the proof of \Cref{mainTheorem} exploits \Cref{ShiodaThm} in an essential way.} On the other hand, when the morphism $p$ is ramified and we have a non-torsion section $\sigma:B\rightarrow \cE$, as a consequence of \Cref{mainTheorem} we get $\ker{\rho}\neq \{0\}$. Thus, asking about the shape of loops whose homotopy class lies in $\ker{\rho}$ is a rather natural question. Here, we want to say something more about the group $\ker{\rho}$ and obtain an `effective' result on its shape in terms of loops.

Consider the Legendre elliptic scheme $\cL \rightarrow S$ and denote by $R=\{r_1, \ldots, r_k\}$ the branch locus of the morphism $p$. Denote the generators of the free group $\pi_1(S- R,s)$ by $\mathfrak{a}_0, \mathfrak{a}_1, \mathfrak{d}_i$: to be more precise, $\mathfrak{a}_0, \mathfrak{a}_1$ are homotopy classes of small loops around $0,1$ respectively; $\mathfrak{d}_i$ is the homotopy class of a small loop around $r_i$. We will denote by the symbol $\langle Y \rangle$ the free group generated by a set $Y$. Let's put
\begin{equation}\label{D0}
X^{(0)}:=\{\mathfrak{d}_i : i =1, \ldots, k\}, \qquad D^{(0)}:=\langle X^{(0)} \rangle
\end{equation}
and define by recursion
$$
X^{(n)}:=\{\mathfrak{a}\mathfrak{d}\mathfrak{a}^{-1}\mathfrak{d}^{-1} : \mathfrak{a}\in\{\mathfrak{a}_0,\mathfrak{a}_1\}, \mathfrak{d}\in D^{(n-1)}\}, \qquad D^{(n)}:=\langle \bigcup_{i=0}^nX^{(i)} \rangle.
$$
Define $K:=\langle \bigcup_{n \in \bN} D^{(n)} \rangle$.

\begin{thm}\label{commutators}
	We have
	$$
	\ker{\rho_B}=p_*^{-1}(K).
	$$
\end{thm}

\begin{proof}
	Let us consider the inclusion $i:S- R \hookrightarrow S$ and the induced homomorphism $i_*$ of fundamental groups. We claim that
	$$
	i_{*}^{-1}(1)=K.
	$$
	First of all note that $D^{(0)} \subseteq i_{*}^{-1}(1)$. By induction, we easily get $D^{(n)} \subseteq i_{*}^{-1}(1)$ for any $n \in \bN$. Hence we have the inclusion $K \subseteq i_{*}^{-1}(1)$.
	
	Now, we prove the reverse inclusion. Observe that each element $g \in \pi_{1}(S- R,s)$ is a word in the alphabet
	$$
	\mathfrak{A}:=\{1, \mathfrak{a}_{0}, \mathfrak{a}_{1}, \mathfrak{d}_{1}, \ldots, \mathfrak{d}_{k}, \mathfrak{a}^{-1}_{0}, \mathfrak{a}^{-1}_{1}, \mathfrak{d}^{-1}_{1}, \ldots, \mathfrak{d}^{-1}_{k}\}.
	$$
	Define the subset $\mathfrak{A}':=\{\mathfrak{a}_{0}, \mathfrak{a}_{1}, \mathfrak{a}^{-1}_{0}, \mathfrak{a}^{-1}_{1}\}$. We write down words without using exponentiation, so that when writing $g$ as a word we intend that each factor lies in $\mathfrak{A}$. Now, consider a reduced word $g$ and delete all elements of the type $\mathfrak{d}_{j}$; label the remaining elements by $\mathfrak{e}_{1}, \ldots, \mathfrak{e}_{m}$, where each $\mathfrak{e}_{j}$ is an element of $\mathfrak{A}'$. Note that the element $i_{*}(g)$ is obtained by $g$ exactly deleting all elements of the type $\mathfrak{d}_{j}$.
	
	From now on, we only consider elements $g,h \in i_{*}^{-1}(1)$. Since $\pi_{1}(S,s)$ is a free group, the condition $g \in i_{*}^{-1}(1)$ implies that $m$ is even and moreover there exists a permutation $\tau$ of the set $\{1, \ldots, m\}$ such that $\mathfrak{e}_{\tau(j)}$ is the inverse of $\mathfrak{e}_{j}$. Such a permutation is not unique, anyway using again that $\pi_{1}(S,s)$ is a free group we can assume that $\tau$ satisfies the following property: the reduced word $g$ is of the form $\mathfrak{d}\mathfrak{e}_{1}\mathfrak{h}\mathfrak{e}_{\tau(1)}\mathfrak{f}$ where $\mathfrak{h},\mathfrak{f} \in i_{*}^{-1}(1)$.
	
	We want to prove that $i_{*}^{-1}(1)\subseteq K$. To this end, define the integer $n:=m/2$ for $g \in i_{*}^{-1}(1)$ and proceed by induction over $n$; the integer $n$ will be called \emph{relative length of $g$}. If $n=0$, then $g \in D^{(0)}$ and we are done. Now, suppose that if $h$ is a reduced word of relative length $< n$ lying in $i_{*}^{-1}(1)$ then $h \in D^{(n-1)}$ and consider a reduced word $g$ of relative length $n$. Let $\mathfrak{a}$ be the first element of $\mathfrak{A}'$ which appears in the word $g$ in left-right order, i.e. $\mathfrak{a}:=\mathfrak{e}_{1}$ in our previous notation. Then, thanks to the above considerations when defining $\tau$, we can state that the element $g$ is of the form
	$$
	\mathfrak{d}\mathfrak{a}\mathfrak{h}\mathfrak{a}^{-1}\mathfrak{f},
	$$
	where $\mathfrak{h}, \mathfrak{f}$ are reduced words of relative length $< n$ lying in $i_{*}^{-1}(1)$ and moreover $\mathfrak{a} \in \mathfrak{A}'$, $\mathfrak{d} \in D^{(0)}$. By inductive hypothesis, we get $\mathfrak{h}, \mathfrak{f} \in D^{(n-1)}$. With simple manipulations we get
	$$
	g=\mathfrak{d}(\mathfrak{a}\mathfrak{h}\mathfrak{a}^{-1}\mathfrak{h}^{-1})\mathfrak{h}\mathfrak{f},
	$$
	which proves $g \in D^{(n)}$ and consequently the claim. By \Cref{periodFunctions} we finally obtain
	$$
	\ker{\rho_B}=p_*^{-1}(K).
	$$	
\end{proof}

\Cref{commutators} gives a way of finding explicit loops in $B$ which leaves periods unchanged: we obtain them as liftings of some loops in $S$ whose homotopy classes lie in $K$. Clearly, not all the elements of $K$ provide loops in $B$: in general such a lifting is a path. However, making use of considerations about ramification of the morphism $p$ we can always find some of them. We exhibit some of them in the next sections.

Before closing this section, we prove an invariance result in the spirit of \Cref{baseChange}. In other terms, the following lemma states that the effectivity of \Cref{mainTheorem} is invariant under finite base change. However, we will be more precise about the meaning of `effectivity of \Cref{mainTheorem}' in the next sections, which entirely concern this topic.

\begin{lemma}\label{baseChangeEffective}
	Let $\varphi: \widetilde{B} \rightarrow B$ be a finite morphism and let $\widetilde{\cE}$ be the pullback of $\cE\rightarrow B$ via $\varphi$. Let $\widetilde{\alpha}$ be a loop in $\widetilde{B}$ whose homotopy class $g$ is such that $g \in \ker{\widetilde{\rho}}$ and $\theta_{\varphi^{*}(\sigma)}(g)\neq 0$. Then $\alpha:=\varphi\circ\widetilde{\alpha}$ is a loop in $B$ whose homotopy class $h$ satisfies $h \in \ker{\rho}$ and $\theta_{\sigma}(h)\neq 0$.
\end{lemma}

\begin{proof}
	This is straightforward by \Cref{eqn3_bis} and \Cref{eqn4}.\\
\end{proof}

\section{Effective proof}\label{effectiveProofSection}

In this section we prove the main (effective) result of the paper. To this end, we focus on \Cref{mainTheorem}. Once we have proved that $M_\sigma^\textnormal{rel}$ is non-trivial we could get the complete theorem by a purely algebraic argument losing a little effectivity in this final step; anyway, in the last part of the paper we provide an alternative way of proving that $M_{\sigma}^{\textnormal{rel}}\cong \bZ^{2}$ which has the advantage of being effective (see \Cref{mainTheoremStrong}). Thus, here we focus on the crucial step of the proof, that is proving $M_\sigma^\textnormal{rel}$ is non-trivial. In other words, the statement $M_\sigma^\textnormal{rel}\neq \{0\}$ means that there exists a loop in $B$ along which periods can be continuously defined but logarithm cannot be defined.

Our strategy is \emph{to explicitly find such a loop $\Gamma$ in $B$ with a certain base point $b_1$, along which the logarithm changes value in $b_1$ but periods don't}, which gives an effective answer to the problem.

In order to be more precise, assuming we have an elliptic scheme $\cE \rightarrow B$ which is pullback of the Legendre elliptic scheme by a finite map $p:B \rightarrow \bP_{1}-\{0,1,\infty\}$, we are going to construct a suitable finite subset $\{b_1, \ldots, b_{N}\}$ of $B$ and to prove the following result:

\begin{thm}\label{mainTheoremTrue}
	Let $\sigma:B \rightarrow \cE$ be a non-torsion section of the elliptic scheme. There exists a loop $\Gamma$ in $B$ which induces trivial monodromy action on periods but non-trivial monodromy action on the logarithm; actually, there exists an index $i \in \{1, \ldots, N\}$ such that $\Gamma$ is of the form
	$$
	\Gamma=\alpha_{1}\delta_{i}\alpha_{i}^{-1}\delta_{i}^{-1},
	$$
	where $\alpha_{1}$ is a loop whose homotopy class is one of the generators of the fundamental group of $B$ (with base point $b_{1}$), $\alpha_{i}$ is a loop based at $b_{i}$ which satisfies $p\circ \alpha_{i} = p\circ \alpha_{1}$, and $\delta_{i}$ is a path from $b_{1}$ to $b_{i}$ which leaves periods unchanged via analytic continuation. In particular, the relative monodromy group of logarithm of $\sigma$ with respect to periods of $\cE \rightarrow B$ is non-trivial, hence isomorphic to $\bZ^2$.
\end{thm}

\subsection{Setup of the proof}
We keep all the notations introduced above. Morever, we shall always use the symbol $\mathcal{B}$ to denote the compactification of an affine curve $B$. We will do a little abuse of notation by denoting with the same letter a morphism $p:B\rightarrow S$ between affine curves and its extension $p:\mathcal{B} \rightarrow \bP_1$. We denote by $\alpha, \beta, \ldots$ paths and loops in $B$, while we denote by $\mathfrak{a}, \mathfrak{b}, \ldots$ their projection in $S$ via the morphism $p$. In particular, we denote by $\mathfrak{a}_0$ and $\mathfrak{a}_1$ two fixed and small enough circles centered at $0$ and $1$, respectively. Note that we can choose all paths in $B$ and $S$ such that they avoid all the ramification points and the branch points of $p: B\rightarrow S$, respectively.

Let's start with some preliminary considerations. Fix a non-torsion section $\sigma$ of an elliptic scheme $\cE \rightarrow B$. By \Cref{twoTors}, \Cref{baseChange} and \Cref{baseChangeEffective}, in what follows we can always assume that $\cE \rightarrow B$ is obtained as pullback of $\cL \rightarrow S$ via a finite Galois morphism $p: B\rightarrow S$. In this case, the section $\sigma$ is also called an algebraic section of $\cL \rightarrow S$, i.e. we have the following diagram
\begin{equation}\label{diagramScheme}
	\begin{tikzcd}
		& & & \cE \arrow{r} \arrow{d} & \cL \arrow{d}\\
		\bH \arrow{r} & B_\sigma \arrow{r} & B^* \arrow{r}  & B \arrow{r}{p} \arrow[bend left]{u}{\sigma} & S,
	\end{tikzcd}
\end{equation}
where we are referring to the tower of covers in \Cref{monodromyDiagram}. By \Cref{ShiodaThm} the section $\sigma$ is ramified, in the sense that the morphism $p: B\rightarrow S$ introduced above is ramified. However, we look at the branch points of $\sigma$ in $\mathbb{P}_1$ admitting also $0,1,\infty$ as possible branch points, i.e. we are interested in the branch points of $p:\mathcal{B} \rightarrow \mathbb{P}_1$. We will denote the branch locus of $p:\mathcal{B} \rightarrow \mathbb{P}_1$ by $\mathcal{R}$ and the branch locus of $p:B \rightarrow S$ by $R:=\{r_1, \ldots, r_k\}$; thus we have $R=\mathcal{R} \cap S$ and $R\neq \emptyset$. Fix a base point $s \in S-R$ and a base point $b=b_1 \in p^{-1}(s)$; denote by $b_i$ the elements of $p^{-1}(s)$. Since the section is algebraic, then each branch point has finite branching order.

\begin{rem}
	This remark simply to briefly clarify why $p$ can be assumed to be Galois. We call $p$ a \emph{Galois (branched) cover} if the morphism $p:\mathcal{B}- p^{-1}(\mathcal{R}) \rightarrow \bP_{1}- \mathcal{R}$ is a Galois unramified cover. By \cite[Theorem 3.3.7]{Sz} this corresponds to having a Galois extension of function fields $\bC(S) \subseteq \bC(B)$. Given an elliptic scheme, by \Cref{twoTors}, \Cref{baseChange} and \Cref{baseChangeEffective}, in a first moment we can simply assume to have a finite morphism $B\rightarrow S$ which realizes our elliptic scheme as pullback of the Legendre scheme. Generally, the morphism $B\rightarrow S$ is ramified and is not Galois. Anyway, the Galois closure of the field extension $\bC(S) \subseteq \bC(B)$ is a finite extension of $\bC(S)$. Thus, appealing to \Cref{baseChange} again, we can always assume $\cE \rightarrow B$ to be obtained as pullback of $\cL \rightarrow S$ via a finite Galois morphism $p: B\rightarrow S$.
\end{rem}

\paragraph{Plan of the proof.}
We explain how the proof is organized in order to make the reader aware of what we are doing in the different steps:

\begin{enumerate}
	\item[1.] \textbf{Control on ramification locus of sections:} this is the fundamental step of our construction. A method for controlling the ramification of sections is provided and it allows to lift loops between covering maps getting control on the monodromy action on periods and logarithms.
		
	\item[2.] \textbf{Explicit loop:} this step regards explicit construction of $\Gamma$. To this end we need some auxiliary paths $\delta_i$ and loops $\alpha_i$ which induce a controlled monodromy action on periods and logarithm. The auxiliary paths will be obtained by using the control on ramification deriving from the previous step, the monodromy action of a covering map, Mordell-Weil theorem for elliptic schemes and Shioda's theorem.
		
	\item[3.] \textbf{Monodromy action and analytic continuation:} a careful analysis of analytic continuation of the relevant functions concludes the proof.
\end{enumerate}

\subsection{Control on ramification locus of sections}
The first step of our proof consists in finding a way to ``control the branch locus $\mathcal{R}$'' (the formal meaning of this sentence is clarified in the next lemma). The right way of carrying out this step has been found in a factorization of the morphism $p:\mathcal{B} \rightarrow \bP_1$ which appears in \cite{CZ2}.\footnote{The authors use this factorization in \cite{CZ2} in a different way with respect to us. They use the cohomological interpretation of variations of logarithm and they need using holomorphic (not merely algebraic) sections. Our approach involves analysis on the ramification locus of the section.}

\begin{lemma}\label{controlBranch}
	Given the ramified Galois morphism $p: \mathcal{B} \rightarrow \bP_1$ associated to the elliptic scheme $\cE \rightarrow B$ as in \Cref{diagramScheme}, there exist a smooth curve $\mathcal{B}'$ and finite morphisms $p_1:\mathcal{B} \rightarrow \mathcal{B}'$, $p_2:\mathcal{B}' \rightarrow \bP_1$ such that the map $p$ factorizes in the following way
	$$\begin{tikzcd}
		\mathcal{B} \arrow{r}{p_1} \arrow[bend right, swap]{rr}{p} & \mathcal{B}' \arrow{r}{p_2} & \bP_1,
	\end{tikzcd}$$
	where
	\begin{enumerate}
		\item[$(i)$] $p_*(\pi_1(B,b))=p_{2*}(\pi_1(B',b'))$ for base points $b \in B$ and $b':=p_1(b) \in B'$;
		
		\item[$(ii)$] $p_1$ is a ramified Galois morphism with branch locus $A$ satisfying $R \subseteq p_2(A) \subseteq \mathcal{R}$;
		
		\item[$(iii)$] $p_2$ is a ramified morphism with branch locus contained into the set $\{0,1,\infty\}$.
	\end{enumerate}
\end{lemma}

\begin{proof}
	Fix base points $b \in B- p^{-1}(R)$ and $s:=p(b) \in S- R$. The continuous map $p:B \rightarrow S$ induces a homomorphism of the corresponding fundamental groups, i.e.
	$$
	p_*:\pi_1(B,b) \rightarrow \pi_1(S,s)=\Gamma_2.
	$$
	Define $\overline{G}:=p_*(\pi_1(B,b))$. The inclusion $\overline{G} \subseteq \Gamma_2$ induces an unramified covering map $p_2: B' \rightarrow S$ between Riemann surfaces such that $p_{2*}(\pi_1(B',b'))=\overline{G}$, where $b'$ is a fixed base point in $B'- A$ which satisfies $p_2(b')=s$. Since $p_*(\pi_1(B,b))=p_{2*}(\pi_1(B',b'))$, there exists a morphism $p_1:B\rightarrow B'$ such that $p_1(b)=b'$. Considering the corresponding elliptic schemes we obtain the diagram
	\begin{equation}\label{factorScheme}
		\begin{tikzcd}
			\cE \arrow{d} \arrow{r} & \cL' \arrow{d} \arrow{r} & \cL \arrow{d}\\
			B \arrow{r}{p_1} \arrow[bend right, swap]{rr}{p}& B' \arrow{r}{p_2} & S.
		\end{tikzcd}
	\end{equation}
	We extend each morphism to the compactification of the curves. Part $(i)$ is straightforward by construction. Observe that part $(iii)$ is also true by construction, since $p_2:B'\rightarrow S$ is unramified. Moreover, note that $p_1$ is Galois since $p$ is. Now, call $A$ the branch locus of $p_1$ and observe that
	$$
	\textnormal{Branch}(p)=\textnormal{Branch}(p_2) \cup p_2(\textnormal{Branch}(p_1)),
	$$
	where the meaning of the just introduced notation `$\textnormal{Branch}$' is clear from the context. Hence, we get $R \subseteq p_2(A) \subseteq \mathcal{R}$.\\
\end{proof}

From now on, we fix base points $b \in B- p^{-1}(R)$, $b':=p_1(b) \in B'- A$ and $s:=p_2(b') \in S-R$. We will write
$$
p_1^{-1}(b')=\{b_1, \ldots, b_N\},
$$
where $b$ is identified with $b_1$.

Now, we want to determine some paths $\delta_i$ which join the points $b_i$ between them and such that their projections $p(\delta_i)$ in $S$ are homotopically trivial. This last condition is essential and is equivalent to saying that periods does not have monodromy along the paths, in the sense that $c_{\delta_{i}}(\omega)=\omega(\delta_{i}(0))$ for any period function $\omega$. To make this possible, we use the monodromy action of $\pi_1(\mathcal{B}'-A,b')$ on the fiber $p_1^{-1}(b')$, but we need enough control on the branch locus $\mathcal{R}$ which is provided by \Cref{controlBranch}.

\begin{lemma}\label{delta_ij}
	For each $i,j \in \{1, \ldots, N\}$ there exists a path $\delta_{i,j}:I \rightarrow B-p_1^{-1}(A)$ such that
	\begin{enumerate}
		\item $\delta_{i,j}(0)=b_i$ and $\delta_{i,j}(1)=b_j$;
		
		\item $p(\delta_{i,j})$ is a homotopically trivial loop in $S$.
	\end{enumerate}
\end{lemma}

\begin{proof}
	Let us consider the unramified cover $$p_1:\mathcal{B}-p_{1}^{-1}(A) \rightarrow \mathcal{B}'-A$$ which is a Galois cover. Therefore, if we denote by $\textrm{Bij}(*)$ the set of all bijections on a set, we have a monodromy representation $$\pi_1(\mathcal{B}'-A,b') \rightarrow \textrm{Bij}(p_1^{-1}(b'))=\textrm{Bij}(\{b_1, \ldots, b_N\}),$$ where the image is a transitive subgroup of $S_N$. In other words this amounts to saying that for each $i,j$ there exists a loop $\mu'_{i,j}$ in $\mathcal{B}'-A$ with base point $b'$ which lifts to a path $\mu_{i,j}:I \rightarrow \mathcal{B}-p_{1}^{-1}(A)$ such that $\mu_{i,j}(0)=b_i, \mu_{i,j}(1)=b_j$.
	
	Denote by $[*]$ the homotopy class of a loop. By part $(i)$ of \Cref{controlBranch}, there exists a loop $\xi$ in $B$ such that $p_*([\xi])=p_{2*}([\mu'_{i,j}])$. Further, $\xi$ defines a loop $\xi':=p_1(\xi)$ in $B'$ such that $p_{2*}([\xi']) = p_{2*}([\mu'_{i,j}])$. Since the monodromy representation $\rho_{B'}$ is faithful, we get that $\mu'_{i,j}$ and $\xi'$ are homotopically equivalent in $B'$. Define $\delta_{i,j}:=\xi^{-1}*\mu_{i,j}$. By construction, the loop $p(\delta_{i,j})$ is homotopically trivial in $S$ and satisfies $\delta_{i,j}(0)=b_i$ and $\delta_{i,j}(1)=b_j$.
	
	Observe that all paths and loops can be chosen avoiding all the ramification and branch points of the maps. Therefore, $\delta_{i,j}$ can be viewed as a loop in $B-p_{1}^{-1}(A)$.\\
\end{proof}

\subsection{Explicit loop}

From now on, we fix the paths $\delta_{i-1,i}$ for each $i=2, \ldots,N$ coming from \Cref{delta_ij}, and define the paths $\delta_i$ as follows:
\begin{equation}\label{delta_i}
	\delta_i:= \delta_{1,2} *\cdots *\delta_{i-1,i}.
\end{equation}
In other words, $\delta_i$ is a path joining $b_1$ to $b_i$ such that $p(\delta_i)$ is a homotopically trivial loop in $S$ with base point $s$. Let's consider a branch $\log_\sigma(b_1)$ of logarithm of $\sigma$ at $b_1$ and define $\log_\sigma$ at the points $b_i$ by analytic continuation along the path considered in \Cref{delta_i}, i.e.
\begin{equation}\label{defLogb_i}
	\log_\sigma(b_i):=c_{\delta_i}(\log_\sigma).
\end{equation}
Consider the map $p_1$ defined in \Cref{factorScheme} and denote by $\cE(B)$ the group of rational sections of the elliptic scheme. Now, we refer to the trace operator $\textrm{Tr}_{p_1}: \cE(B) \rightarrow \cL'(B')$ defined as follows: for each $t \in B$ and $t'=p_1(t) \in B'$, the elliptic curves $\cE_t$ and $\cL'_{t'}$ are canonically identified; then, for each section $\sigma:B \rightarrow \cE$ define $\textrm{Tr}(\sigma)= \textrm{Tr}_{p_1}(\sigma): B' \rightarrow \cL'$ by setting
\begin{equation}\label{TraceOperator}
	\textrm{Tr}(\sigma)(t') = \sum_{t \in p_1^{-1}(t')} \sigma(t) \in \cL'_{t'},
\end{equation}
where the sum is taken counting multiplicities. Note that $\textrm{Tr}(\sigma)$ is a rational section of the modular elliptic scheme $\mathcal{L}'\rightarrow B'$, thus by \Cref{ShiodaThm} it is a torsion section. Therefore, by \Cref{relMonGrp} we can multiply the section by a non-zero integer and assume $\textrm{Tr}_{p_1}(\sigma)=0$; we continue denoting by $\sigma$ the new section. In other words, this means that
$$
\sigma(b_1) + \cdots + \sigma(b_N)=0,
$$
which is equivalent to saying
$$
\log_\sigma(b_1) + \cdots + \log_\sigma(b_N)=\omega,
$$
where $\omega$ is a period.

\begin{lemma}\label{assumptionLog}
	Up to multiplying the section by a non-zero integer and to changing determination of $\log_\sigma$, we can suppose
	$$\log_\sigma(b_1) + \cdots + \log_\sigma(b_N)=0.$$
\end{lemma}

\begin{proof}
	Let us consider
	$$\ell(b_1):=N\cdot \log_\sigma(b_1) - \omega,$$
	where $N$ is the cardinality of $\{b_1, \ldots, b_N\}$; observe that $\ell(b_1)$ is a logarithm of the section $N\cdot \sigma$ over the point $b_1$ and that this section still satisfies the condition $\textrm{Tr}_{p_1}(N\sigma)=0$. Since periods have no monodromy along the paths $\delta_i$, by analytic continuation as above we obtain
	$$\ell(b_i):=c_{\delta_i}(\ell)=N\cdot \log_\sigma(b_i)-\omega,$$
	Therefore, we get
	$$\ell(b_1) + \cdots + \ell(b_N) = N\cdot[\log_\sigma(b_1) + \cdots + \log_\sigma(b_N)-\omega]=0.$$
\end{proof}

In view of \Cref{relMonGrp}, it is enough to prove \Cref{mainTheoremTrue} for a multiple of the section with a suitable fixed branch of its elliptic logarithm. Thus, from now on we replace our section by a suitable multiple but we keep the notation $\sigma$. Moreover, we fix the branch $\ell$ at $b_1$ which satisfies \Cref{assumptionLog}. In other words, we can assume that the section $\sigma$ admits a branch $\ell$ of logarithm which satisfies
\begin{equation}\label{conditionLog}
	\ell(b_1) + \cdots + \ell(b_N)=0.
\end{equation}

Generally, analytic continuation of $\ell$ along a loop $\zeta$ gives us a new branch of logarithm, which differs from the first one by a period. In other words we have
$$
c_{\zeta}(\ell) = \ell(b_1)+\omega_{\zeta},
$$
where $\omega_{\zeta}$ is a period which may also be zero. In order to obtain our explicit loop we need some control on the variation $\omega_\zeta$ for some loop $\zeta$. By \Cref{wellDefLog}, for some loop $\zeta$ (with base point $b_1$) we can assume $\omega_\zeta\neq 0$. Notice that if this loop leaves unchanged the periods, we just have the conclusion of our theorem. Anyway, this would not be enough to provide effectivity nor to allow the proof to work. In order to really obtain effectivity, observe that the loop $\zeta$ can be chosen into a set of finitely many loops, i.e. the generators of the fundamental group $\pi_1(B,b_1)$ where $B$ is a genus $g$ Riemann surface minus the finite set of points $p^{-1}(\mathcal{R})$. From now on, we denote by $\alpha$ a generator of $\pi_1(B,b_1)$ such that
\begin{equation}\label{alpha}
	c_{\alpha}(\ell) = \ell(b_1)+\omega_{\alpha},
\end{equation}
where $\omega_{\alpha}$ is a non-zero period. Now we introduce some auxiliary loops:

\begin{enumerate}
	\item let us denote by $\alpha'$ the loop in $B'-A$ with base point $b'$ defined by $\alpha':=p_1(\alpha)$ and by $\mathfrak{a}$ the loop in $S-R$ with base point $s$ obtained as the projection $p(\alpha)$;
	
	\item since $p_1:B \rightarrow B'$ is a Galois morphism and since one of the liftings of $\alpha'$ via $p_1$ is the closed path $\alpha$, then all the liftings of $\alpha'$ via $p_1$ are closed paths. In other words, $\alpha'$ can be lifted by $p_1$ to a closed path $\alpha_i$ in $B-p_{1}^{-1}(A)$ with base point $b_i$ for each $i=1, \ldots, N$; with this notation we have $\alpha_1=\alpha$. Observe that $p(\alpha_i)=\mathfrak{a}$ for each $i$.
\end{enumerate}

By considering analytic continuation along $\alpha_i$ we obtain
$$
c_{\alpha_i}(\ell) = \ell(b_i) + \omega_{\alpha_i},
$$
where $\omega_{\alpha_i}$ is a period which may also be zero. Thanks to our previous considerations we have some control on the variation $\omega_{\alpha_1}$, but we need something more. To this end, the next lemma provides enough control on the variations $\omega_{\alpha_i}$ too.

\begin{lemma}\label{control_omegai}
	There exists an index $i \in \{1,\ldots,N\}$ such that
	$$
	\omega_{\alpha_i}\neq \omega_{\alpha_1}.
	$$
\end{lemma}

\begin{proof}
	Recall that $\alpha_i$ are liftings of the same loop $\alpha'$. Since $\textrm{Tr}(\sigma)(\alpha'(t))\equiv 0$ for each $t$, by evaluating $\ell$ along the loops $\alpha_i$, we have
	$$
	\sum_i \ell(\alpha_i(t)) = \omega(t) \textrm{ for each } t \in I,
	$$
	where $\omega(t):=n(t)\omega_1(t) + m(t)\omega_2(t)$ is a period function with $n(t),m(t) \in \bZ$; $\omega_1(t), \omega_2(t)$ is as usual a holomorphic basis for period lattices in the corresponding points. Since the map $t\mapsto (n(t),m(t))$ is continuous and $\bZ^2$ is discrete we obtain that $(n(t),m(t))$ is constant with respect to $t$. Since $(n(0),m(0))=(0,0)$ by $(\ref{conditionLog})$, we deduce that
	$$
	\omega(t)\equiv 0 \textrm{ for each } t \in I.
	$$
	Since $\ell(\alpha_i(1))=\ell(\alpha_i(0)) + \omega_{\alpha_i}$, we obtain
	$$
	\sum_i\ell(b_i) + \sum_i \omega_{\alpha_i} = \sum_i \omega_{\alpha_i} = 0.
	$$
	The conclusion follows by the conditions $\omega_{\alpha_1}\neq0$ and $\sum_i \omega_{\alpha_i}= 0$.\\
\end{proof}

\subsection{Monodromy action and analytic continuation}
Appealing to \Cref{periodFunctions}, we define period functions $\omega_1, \omega_2$ in a connected and simply connected neighborhood $U$ of $b \in B-p^{-1}(R)$ by
$$
\omega_1=\omega_1^{\mathcal{L}}\circ p, \qquad \omega_2=\omega_2^{\mathcal{L}}\circ p.
$$
This argument can be retraced on an open covering of any path in $B-p^{-1}(R)$:  this means that the monodromy of the periods of $\mathcal{E} \rightarrow B$ is non trivial only along the paths whose projection via $p:B \rightarrow S$ is a loop in $S$ which is not homotopically trivial. Hence, by \Cref{periodFunctions} and \Cref{delta_ij} we get
$$
\omega(b_i):=c_{\delta_i}(\omega)=\omega(b_1)
$$
for each period function $\omega$ locally defined on $B$, and
$$
c_{\alpha_i}(\omega_{1})=c_\mathfrak{a}(\omega_{1}^\cL), \qquad c_{\alpha_i}(\omega_{2})=c_\mathfrak{a}(\omega_{2}^\cL).
$$
In other words, the monodromy of $\omega_1, \omega_2$ along a path in $B$ is the same as the monodromy of $\omega_1^{\mathcal{L}}, \omega_2^{\mathcal{L}}$ along the projection of the path via the morphism $p:B \rightarrow S$.

Looking at \Cref{delta_i} and \Cref{control_omegai}, let us fix an index $i$ which satisfies $\omega_{\alpha_i}\neq \omega_{\alpha_1}$ and consider the path $\delta_i$. Define the loop 
\begin{equation}\label{Gamma}
	\Gamma:=\alpha_1\delta_i\alpha_i^{-1}\delta_i^{-1}.
\end{equation}

The monodromy of any period along $\Gamma$ is determined by our previous considerations. In fact, for any period function $\omega$ we have
$$
c_{\Gamma}(\omega)=c_{\delta_i^{-1}}c_{\alpha_i^{-1}}c_{\delta_i}c_{\alpha_1}(\omega)= c_{\alpha_1^{-1}} c_{\alpha_1}(\omega)= \omega(b_1).$$ 
So periods don't change value in $b_1$ along $\Gamma$. Now, we want to study the analytic continuation of logarithm along $\Gamma$.

\begin{lemma}\label{lemmaMonodromy}
	We have the following relation:
	$$c_{\alpha_i^{-1}}(\log_\sigma)  = \log_\sigma(b_i) -c_{\alpha_i^{-1}}(\omega_{\alpha_i}).$$
\end{lemma}

\begin{proof}
	
	By linearity of $c_\alpha$ we have
	\begin{align*}
		\log_\sigma(b_i)&=c_{\alpha_i^{-1}}c_{\alpha_i}(\log_\sigma)=c_{\alpha_i^{-1}}(\log_\sigma+\omega_{\alpha_i})=\\
		&=c_{\alpha_i^{-1}}(\log_\sigma)+c_{\alpha_i^{-1}}(\omega_{\alpha_i}).
	\end{align*}
\end{proof}
Therefore, we finally get
\begin{equation}\label{finalCalc}
	\begin{aligned}
	\log_\sigma(b_1) &\xrightarrow{\alpha_1} \log_\sigma(b_1)+ \omega_{\alpha_1} \xrightarrow{\delta_i} \log_\sigma(b_i) + \omega_{\alpha_1}\\
	&\xrightarrow{\alpha_i^{-1}} \log_\sigma(b_i) - c_{\alpha_i^{-1}}(\omega_{\alpha_i}) + c_{\alpha_i^{-1}}(\omega_{\alpha_1})\\
	&\xrightarrow{\delta_i^{-1}} \log_\sigma(b_1) + c_{\alpha_i^{-1}}(\omega_{\alpha_1} - \omega_{\alpha_i}).
	\end{aligned}
\end{equation}
Thus, the value of the logarithm in $b_1$ has changed by a period $c_{\alpha_i^{-1}}(\omega_{\alpha_1} - \omega_{\alpha_i})$. Since we are supposing that $\omega_{\alpha_1} - \omega_{\alpha_i} \neq 0$, then the above period is not zero: recall that $c_{\alpha_i}$, and so $c_{\alpha_i^{-1}}$, is injective. In conclusion, we have explicitly determined a loop $\Gamma$ which establishes
$$
M_{\sigma}^{\textnormal{rel}}\neq \{0\}.
$$
It's a purely algebraic matter to prove that $M_\sigma^\textnormal{rel} \cong \bZ^2$: it follows by the irreducibility of the action of $\rho(G)$ on the lattice of periods (see \cite{CZ2} for details). Here this last step is performed without effectivity because we are mainly interested in proving non-triviality of the monodromy group in a first moment, anyway an effective version of it is provided by \Cref{mainTheoremStrong}.

\begin{rem}
	Note that the proof is really effective: we obtained a way of determining explicitly in finite time a loop $\Gamma$ with the properties we need. In fact, each path appearing in \Cref{Gamma} can be found in a finite number of attempts: $\alpha_{1}$ is a representative of one of the generators of $\pi_1(B,b_{1})$; $\alpha_{i}$ is one of the finitely many liftings of $\alpha'$; $\delta_{i}$ is determined by the monodromy action of $\pi_{1}(\mathcal{B}'-A,b')$ on the fiber $p_{1}^{-1}(b')$ and by the ramification of the section.
\end{rem}

\section{A variation in the approach: Abhyankar's Lemma and Belyi pairs}

We propose a second approach to the problem with the aim of simplifying the shape of the loop $\Gamma$. This approach works by adding a hypothesis on the ramification locus of the section but it has the advantage to give a more immediate way of constructing the paths $\delta_i$ considered above. Thus, the theorem we want to prove is the following:

\begin{thm}\label{AbhyankarApproach}
	Let $\sigma:B \rightarrow \cE$ be a non-torsion algebraic section of $\cL \rightarrow S$ and suppose that the ramification locus $\mathcal{R}$ of $\sigma$ does not contain the whole set $\{0,1,\infty\}$. Then the loop $\Gamma$ constructed in \Cref{Gamma} can be chosen of the type
	$$
	\Gamma:=\alpha_1\delta_i\alpha_i^{-1}\delta_i^{-1},
	$$
	where the homotopy class of $\alpha_1$ is one of the generators of $\pi_1(B,b)$ and $\delta_i \in p_*^{-1}(D^{(0)})$; the group $D^{(0)}$ is defined in \Cref{D0}.
\end{thm}

The hypothesis on the ramification locus $\mathcal{R}$ means that we are again in the situation of \Cref{diagramScheme} but this time the intersection between the branch locus of $p$ and the set $\{0,1,\infty\}$ contains at most two points. Up to an isomorphism, we can suppose that these (at most) two points are chosen between $\{0,\infty\}$. Thus, the hypothesis can be reformulated by saying that the section is not ramified at $1$.

Thus, we assume that the section is not ramified at $1$ for the rest of the proof and keep all the notations introduced above.

\paragraph{Plan of the proof.}
The steps of the proof follow the same philosophy as before but the control on the ramification locus of the section is reached by different techniques: in this case we get a complete `separation of the branch locus'. We are exposing the proof in two steps:\begin{enumerate}
	\item[1.] \textbf{Separation of the branch locus of sections:} we reduce the proof of the theorem to the case of an elliptic scheme whose morphism $p:B\rightarrow S$ admits a factorization which allows us to have a powerful control on the ramification of the section. Precisely, \Cref{redBettCase} takes the place of \Cref{controlBranch}: this is the essential step that allows us to carry out the construction of the loop $\Gamma$. We will make use of the Abhyankar's lemma, which we state as in \cite[Theorem 12.4.4]{Villa}:
	
	\begin{thm}\label{Abhyankar}\textbf{\textit{(Abhyankar's Lemma)}}
		Let $L/K$ be a finite separable extension of function fields. Suppose that $L=K_1K_2$ with $K\subseteq K_i \subseteq L$. Let $\mathfrak{p}$ be a prime divisor of $K$ and $\mathfrak{P}$ a prime divisor in $L$ above $\mathfrak{p}$. Let $\mathfrak{P}_i=\mathfrak{P}\cap K_i$ for $i=1,2$. If at least one of the extensions $K_i/K$, $i=1,2$, is tamely ramified at $\mathfrak{p}$, then
		$$e_{L/K}(\mathfrak{P}|\mathfrak{p})= \textrm{lcm}\{e_{K_1/K}(\mathfrak{P}_1|\mathfrak{p}), e_{K_2/K}(\mathfrak{P}_2|\mathfrak{p})\},$$
		where $e$ denotes the ramification index.
	\end{thm}
	
	\item[2.] \textbf{Explicit paths:} the remaining steps are analogous to the previous proof. We reach the construction of $\Gamma$ by using the previous step, Mordell-Weil theorem for elliptic schemes and Shioda's theorem. Studying the analytic continuation of periods and logarithm along $\Gamma$ we get the proof.\end{enumerate}

\subsection{Separation of the branch locus of sections}

Recall we are considering a section which is not ramified at $1$. We are denoting the branch locus of $p:\mathcal{B} \rightarrow \mathbb{P}_1$ by $\mathcal{R}$ and the branch locus of $p:B \rightarrow S$ by $R:=\{r_1, \ldots, r_k\}$. Moreover, by \Cref{ShiodaThm} we have $R\neq \emptyset$. The first step of our proof consists in reducing the problem to a case in which there exists a factorization of the morphism $p:\mathcal{B} \rightarrow \bP_1$ which ``separates the branch locus $\mathcal{R}$''. The formal meaning of this sentence is expressed by the following lemma:

\begin{lemma}\label{redBettCase}
	Given the ramified Galois morphism $p: \mathcal{B} \rightarrow \bP_1$, there exist smooth curves $\mathcal{B}', \widetilde{\mathcal{B}}$ and a finite morphism $\psi:\widetilde{\mathcal{B}} \rightarrow \mathcal{B}$ such that the map $\widetilde{p}:=p\circ \psi$ factorizes in the following way
	$$\begin{tikzcd}
		\widetilde{\mathcal{B}} \arrow{r}{\phi} \arrow[bend right, swap]{rr}{\widetilde{p}:=p\circ \psi} & \mathcal{B}' \arrow{r}{q} & \bP_1,
	\end{tikzcd}$$
	where
	\begin{enumerate}
		\item[$(i)$] $p$ and $\widetilde{p}$ are ramified Galois morphisms and have the same branch locus $\mathcal{R}$;
		
		\item[$(ii)$] $q$ is a ramified Galois morphism with branch locus contained into the set $\{0,\infty\}$;
		
		\item[$(iii)$] $\phi$ is a ramified Galois morphism with branch locus $A$ satisfying $q(A)=R$.
	\end{enumerate}
\end{lemma}

\begin{proof}
	\textbf{\textit{(1) Construction of $\widetilde{\mathcal{B}}$:}} Since we are in characteristic $0$, by the primitive element theorem the extension $\mathbb{C}(\bP_1) \subset \mathbb{C}(\mathcal{B})$ is simple. So there exists an algebraic function $f$ such that $$\mathbb{C}(\mathcal{B})=\mathbb{C}(\lambda)(f),$$ where $\mathbb{C}(\bP_1)=\mathbb{C}(\lambda)$. Thus there exists an irreducible polynomial $P(\lambda,w) \in \mathbb{C}(\lambda)[w]$ such that $P(\lambda,f)=0$.
	
	Since $p: \mathcal{B} \rightarrow \bP_1$ is a ramified cover between Riemann surfaces, we can consider a neighbourhood $U$ of $0$ such that $p^{-1}(U)=\cup_i V_i$ where the open sets $V_i's$ are disjoint and biholomorphic to a disk $D$. For each $i$ we have a biholomorphic map $\varphi_i:D \rightarrow V_i$ such that
	$$
	p\circ \varphi_i (\zeta_i)=\zeta_i^{n_i}.
	$$
	
	We now focus on the local ring $\cO_0$ of $0$. Embed  $\mathcal{O}_0$ in its completion $\hat{\mathcal{O}}_0$, which can be identified with the ring of formal power series $\bC[[\lambda]]$ since $\bP_1$ is a smooth curve. By looking for the solutions of $P(\lambda,f)=0$ in an algebraic extension of $\hat{\mathcal{O}}_0$, for any $i$ we obtain a Puiseux series of the form $$\sum_{\nu=k_i}^\infty c_{\nu,i} \lambda^\frac{\nu}{n_i}.$$ that solves the polynomial equation. Let us consider the integer $$n:=\textrm{lcm}\{n_i: i=1, \ldots, r\}.$$
	
	Repeating the same reasoning in a neighbourhood of $\infty$, we obtain biholomorphic maps which can be expressed as $\xi\mapsto \xi^{m_i}$. By looking at the local ring $\mathcal{O}_\infty$, we obtain the corresponding Puiseux series of the form $$\sum_{\nu}^\infty d_{\nu,i} 	\left(\frac{1}{\lambda}\right)^\frac{\nu}{m_i}.$$ Let us consider $$m:=\textrm{lcm}\{m_i: i=1, \ldots, s\}.$$
	
	Now, let us consider the field extension $$\mathbb{C}(\lambda) \subset \mathbb{C}(\lambda)(\psi), \qquad \textrm{where } \psi^{nm}=\lambda.$$ This field extension gives us a morphism $q:\mathcal{B}'\rightarrow \bP_1$ between smooth curves. We can distinguish two cases:
	\begin{enumerate}
		\item[$\bullet$] if $n=m=1$, then we have $q=\textrm{id}:\bP_1 \rightarrow \bP_1$;
		
		\item[$\bullet$] if $nm>1$, then $q:\mathcal{B}'\rightarrow \bP_1$ is a Galois morphism with branch locus $\{0,\infty\}$.
	\end{enumerate}
	
	In every case, passing to the compositum $\mathbb{C}(\mathcal{B})\mathbb{C}(\mathcal{B}')$ we obtain a smooth curve $\widetilde{\mathcal{B}}$ endowed with finite Galois morphisms
	$$
	\psi: \widetilde{\mathcal{B}} \rightarrow \mathcal{B}, \quad \phi: \widetilde{\mathcal{B}} \rightarrow \mathcal{B}', \quad  \widetilde{p}: \widetilde{\mathcal{B}} \rightarrow \bP_1
	$$
	which make the following diagram commutative:
	$$\begin{tikzcd}
		\widetilde{\mathcal{B}} \arrow{r}{\phi} \arrow{d}[swap]{\psi} \arrow{rd}{\widetilde{p}} & \mathcal{B}' \arrow{d}{q}\\
		\mathcal{B} \arrow{r}{p} & \bP_1.
	\end{tikzcd}$$
	
	\noindent \textbf{\textit{(2) Distribution of ramification points:}} It remains to be shown that parts $(i)$ and $(iii)$ of the lemma are true. Let's start by showing that the branch locus $A$ of $\phi$ satisfies the condition $q(A)\cap\{0,\infty\}=\emptyset$. In other words, this means that all the preimages $q^{-1}(0)$ and $q^{-1}(\infty)$ are not contained in $A$.
	
	In what follows we denote by $\widetilde{x}, \widetilde{y}$ any preimage of $0,\infty$ via $\widetilde{p}$, respectively. Let us observe that by the previous construction, we have the following divisibility relations between ramification indexes:
	$$e(\psi(\widetilde{x})|0) \, | \, e(\phi(\widetilde{x})|0), \qquad e(\psi(\widetilde{y})|\infty) \, | \, e(\phi(\widetilde{y})|\infty).$$
	
	We now use Abhyankar's lemma (\Cref{Abhyankar}) with respect to the fields $L=\bC(\widetilde{\mathcal{B}})$, $K_1=\bC(\mathcal{B})$, $K_2=\bC(\mathcal{B}')$, $K=\bC(\bP_1)$. If we take $\mathfrak{P}$ as the maximal ideal of the local ring $\cO_{\widetilde{x}}$, then $\mathfrak{P}_{1}, \mathfrak{P}_{2}$ are identified with the maximal ideals of the local rings $\cO_{\psi\left(\widetilde{x}\right)}$, $\cO_{\phi\left(\widetilde{x}\right)}$ and $\mathfrak{p}$ with the maximal ideal of the local ring $\cO_0$. Thus we obtain
	$$e(\widetilde{x}|0)=\textrm{lcm}\{e(\psi(\widetilde{x})|0), e(\phi(\widetilde{x})|0)\}=e(\phi(\widetilde{x})|0).$$
	If we consider the local ring in $\widetilde{y}$, we analougously obtain
	$$e(\widetilde{y}|\infty)=\textrm{lcm}\{e(\psi(\widetilde{y})|\infty), e(\phi(\widetilde{y})|\infty)\}=e(\phi(\widetilde{y})|\infty).$$
	
	Since the ramification index is multiplicative, we obtain:
	$$e(\widetilde{x}|\phi(\widetilde{x}))=1, \qquad e(\widetilde{y}|\phi(\widetilde{y}))=1.$$ Thus, $\phi$ is not ramified into the points of $\widetilde{p}^{-1}(0)$ and $\widetilde{p}^{-1}(\infty)$.
	
	Finally, we want to determine the branch locus of $\phi$ in the affine part $B'$ of $\mathcal{B}'$. By using the same reasoning as above, starting with the local rings $\cO_{r_i}$ where $r_i$ are the branch points in $R$, we obtain that the branch locus $A$ of $\phi$ is such that $q(A)\cap S=R$. Thus, we just proved that $q(A)=R$. Part $(i)$ follows directly by Abhyankar's lemma.
	
\end{proof}

\begin{rem}\label{RH}
	Let us consider the morphism $q:\mathcal{B}' \rightarrow \bP_1$. During the proof, we saw that when the section is unramified over the set $\{0,\infty\}$, then $\mathcal{B}'=\bP_1$ and $q=\textrm{id}: \bP_1 \rightarrow \bP_1$. In the case where the section is ramified over $0$ or $\infty$, by using Riemann-Hurwitz formula we obtain again that $\mathcal{B}'$ has genus $0$, therefore $\mathcal{B}' \isom \bP_1$. In particular, in the next sections we will use the fact that $\mathcal{B}'-A \isom \bP_1-A$, where $A$ is a finite set of points.
\end{rem}

By restricting to $S$, the previous lemma gives us a chain of morphisms between affine curves
\begin{equation}\label{reductionBetterCase}
	\begin{tikzcd}
		\widetilde{B} \arrow[bend right, swap]{rr}{\widetilde{p}} \arrow{r}{\phi} & B' \arrow{r}{q} & S,
	\end{tikzcd}
\end{equation}
where 
\begin{enumerate}
	\item[$\bullet$] $B'=\mathcal{B}'-q^{-1}(\{0,1,\infty\})$ and $\widetilde{B} =\widetilde{\mathcal{B}}-\phi^{-1}(q^{-1}(\{0,1,\infty\}))$;
	
	\item[$\bullet$] $q:B' \rightarrow S$ is an unramified cover whose extension to compactifications $q:\mathcal{B}' \rightarrow \bP_1$ has branch locus either empty or equals to $\{0,\infty\}$;
	
	\item[$\bullet$] $\phi: \widetilde{\mathcal{B}} \rightarrow \mathcal{B}'$ has branch locus $A$ such that $q(A)=R$;
	
	\item[$\bullet$] $\widetilde{p}: \widetilde{\mathcal{B}}\rightarrow \bP_1$ has branch locus $\mathcal{R}$.
\end{enumerate}

\subsection{Explicit paths}

By \Cref{redBettCase} we can put in the situation described in \Cref{reductionBetterCase} with the just listed properties. We can pullback the Legendre scheme via these morphisms and obtain the diagram
$$\begin{tikzcd}
	\widetilde{\mathcal{E}} \arrow{r} \arrow{d} & \mathcal{L}' \arrow{d}\arrow{r} & \mathcal{L} \arrow{d}\\
	\widetilde{B} \arrow{r} & B' \arrow{r}& S.
\end{tikzcd}$$
By \Cref{baseChange}, \Cref{baseChangeEffective} and \Cref{redBettCase}, we just need to prove \Cref{AbhyankarApproach} for $\widetilde{\mathcal{E}} \rightarrow \widetilde{B}$. From now on we rename $\widetilde{\mathcal{E}}\rightarrow \widetilde{B}$ with $\mathcal{E} \rightarrow B$ and we will put ourselves in the situation expressed by the following diagram
\begin{equation}\label{diagramAbhyankar}
	\begin{tikzcd}
		\cE \arrow{d} \arrow{r} & \cL' \arrow{d} \arrow{r} & \cL \arrow{d}\\
		B \arrow{r}{\phi} \arrow[bend right, swap]{rr}{p}& B' \arrow{r}{q} & S.
	\end{tikzcd}
\end{equation}
and by the properties in \Cref{reductionBetterCase}. Let us fix a base point $s$ in $S-R$ and let us choose a preimage $b' \in q^{-1}(s)$. We will write
$$
\phi^{-1}(b')=\{b_1, \ldots, b_N\}.
$$ 

Now, we want to determine some paths $\delta_i$ which join the points $b_i$ between them and such that their projections $p(\delta_i)$ are homotopically trivial loops in $S$. To make this possible, we use the monodromy action of $\pi_1(\mathcal{B}'-A,b')$ on the fiber $\phi^{-1}(b')$, where $A$ denotes the branch locus of $\phi$.

\begin{lemma}\label{deltaij2}
	For each $i,j \in \{1, \ldots, N\}$ there exists a path $\delta_{i,j}:I \rightarrow B-\phi^{-1}(A)$ such that
	\begin{enumerate}
		\item $\delta_{i,j}(0)=b_i$ and $\delta_{i,j}(1)=b_j$;
		
		\item $p(\delta_{i,j})$ is a homotopically trivial loop in $S$.
	\end{enumerate}
\end{lemma}

\begin{proof}
	Let us consider the unramified cover
	$$
	\phi:\mathcal{B}-\phi^{-1}(A) \rightarrow \mathcal{B}'-A
	$$
	which is a Galois cover. Therefore, if we denote by $\textrm{Bij}(*)$ the set of all bijections on a set, we have a monodromy representation
	$$
	\pi_1(\mathcal{B}'-A,b') \rightarrow \textrm{Bij}(\phi^{-1}(b'))=\textrm{Bij}(\{b_1, \ldots, b_N\}),
	$$
	where the image is a transitive subgroup of $S_N$. In other words this amounts to saying that for each $i,j$ there exists a loop $\delta'_{i,j}$ in $\mathcal{B}'-A$ with base point $b'$ which lifts to a path $\delta_{i,j}:I \rightarrow \mathcal{B}-\phi^{-1}(A)$ such that $\delta_{i,j}(0)=b_i, \delta_{i,j}(1)=b_j$.
	
	By \Cref{RH}, we have that $\mathcal{B}' \isom \bP_1$. Now we want to choose some particular loops in $\mathcal{B}' -A$ whose homotopy classes are generators of $\pi_1(\mathcal{B}' -A,b')$. These loops are constructed in the following steps:
	
	\begin{enumerate}
		\item[$\bullet$] fix $r'_i\in A$ and let $\mathbb{D}_i \subset \mathcal{B}'-A$ be an open neighbourhood of $r'_i$ isomorphic to a punctured disc; moreover, suppose that $\mathbb{D}_i$ is sufficiently small such that $\mathbb{D}_i\cap q^{-1}\{0,1,\infty\} = \emptyset$ and $q(\mathbb{D}_i)$ is a punctured disc in $S-R$ centered at a point of $R$. Choose a loop $\xi$ in $\mathbb{D}_i$ with winding number one around $r'_i$ and with a base point which we denote by $z_i \in \mathbb{D}_i$. Observe that $q(\xi)$ is a loop contained in $q(\mathbb{D}_i)$, thus it is homotopically trivial in $S=\bP_1-\{0,1,\infty\}$.
		
		\item[$\bullet$] Choose a path $\mu$ from the point $b'$ to the point $z_i \in \mathbb{D}_i$; this path can be chosen in such a way that it does not pass through any point of $q^{-1}(\{0,1,\infty\})$.
		
		\item[$\bullet$] The path $\mu\xi\mu^{-1}$ is a loop in $\mathcal{B}'-A$ based on $b'$: we call such a loop a \emph{small loop on $\mathcal{B}'-A$ around $r'_i$}. Since this loop does not pass through any point of $q^{-1}(\{0,1,\infty\})$, it can be seen as a loop in $B'-A$. Since $q(\xi)$ is homotopically trivial in $S$, then $q(\mu\xi\mu^{-1})$ is also homotopically trivial in $S$.
		
	\end{enumerate}
	
	Let us consider a small loop around $r'_i$ for every $i$: they are loops contained in $B'-A$ whose projections via $q:B' \rightarrow S$ are homotopically trivial in $S$; moreover, their homotopy classes form a set of generators for the fundamental group $\pi_1(\mathcal{B}'-A,b')$. This means that the path $\delta_{i,j}$ can be obtained as a lifting of a composition of small loops around some points of $A$. Since these loops does not pass through any point of $q^{-1}(\{0,1,\infty\})$, then $\delta_{i,j}$ can be viewed as a path in $B-\phi^{-1}(A)$. Moreover, by our construction we obtained that $p(\delta_{i,j})$ is a homotopically trivial loop in $S$.\\
\end{proof}

\begin{rem}
	Observe that in \Cref{deltaij2} we have used in an essential way the construction of $\mathcal{B}'$ and the \emph{`separation of the branch locus'}.
\end{rem}

Fix the paths $\delta_{i-1,i}$ for each $i=2, \ldots,N$ coming from \Cref{deltaij2} and define the paths $\delta_i$ as follows:
\begin{equation}\label{delta_i2}
	\delta_i:= \delta_{1,2} *\cdots *\delta_{i-1,i}.
\end{equation}
In other words, $\delta_i$ is a path joining $b_1$ to $b_i$ such that $p(\delta_i)$ is a homotopically trivial loop in $S$ with base point $s$.

At this points, the proof works exactly as above. For the sake of completeness we give a brief recap. First of all we can analytically continue $\log_\sigma$ at the $b_i$'s as in \Cref{defLogb_i} and we can assume again \Cref{conditionLog} to be satisfied. Again, by Mordell-Weil theorem for function fields we can consider a loop $\alpha$ whose homotopy class is one of the generators of $\pi_1(B,b_1)$ as in \Cref{alpha} and define the auxiliary loops $\alpha_i$. By Shioda's theorem, we obtain \Cref{control_omegai} and define the loop $\Gamma$ as in \Cref{Gamma}. With the same calculations as in \Cref{finalCalc} we get the thesis.

\subsection{Some comments about the approach and Belyi pairs}

Here, we discuss the advantages and limitations of this alternative approach compared to the first one. In what follows we are going to refer to the form of the loop $\Gamma$. Notice that the construction of the loops $\alpha_i$'s is the same in both approaches; what is different is the constructions of the $\delta_i$'s. This second approach is advantageous in terms of effectivity: in fact in \Cref{delta_ij} we need to construct some loop $\xi$ in $B$ which annihilates the monodromy of the path $\mu_{i,j}$. In \Cref{deltaij2}, the paths $\delta_{i,j}$ are provided immediately by the monodromy action of $\pi_1(\mathcal{B}'-A,b')$ on the fiber: each lifting which connects $b_i$ to $b_j$ is fine, since all these paths have a projection which is homotopically trivial in $S$.

On the other hand, the main limitation of this approach is clearly the assumption on the branch locus of sections made in \Cref{AbhyankarApproach}. Here an example to underline this limitation showing that the previous proof cannot be extended without new considerations to the case in which $\sigma$ is ramified at the whole bad locus $\{0,1,\infty\}$.

\begin{exe}\label{exampleAbhyankar}
	Let us consider the algebraic section of $\cL \rightarrow S$
	$$\sigma(\lambda)=\left(x(\lambda), \sqrt{x(\lambda)(x(\lambda)-1)(x(\lambda)-\lambda)}\right),$$
	where $x(\lambda)$ is defined by
	$$x(\lambda):=\sqrt{\lambda} + \sqrt{\lambda-1} + \sqrt{\lambda-r},$$
	for a fixed point $r \in S$.
	
	The ramification locus of this section contains the whole set $\{0,1,\infty\}$. Moreover, the section is also ramified at $r$, so it is a non-torsion section.\footnote{Note that we are not applying \Cref{ShiodaThm} this time: it regards the reverse implication. However, what we are saying is also true by general theory.} Let $p:B\rightarrow S$ the base change which makes $\sigma$ to become well-defined and fix a point $s \in S$ which is not in the branch locus of $\sigma$. The points $b_i$ in the fiber $p^{-1}(s)$ correspond to different choices of the square roots which appear in the coordinates of $\sigma$. 
	
	By using \Cref{redBettCase}, we obtain the diagram
	$$\begin{tikzcd}
		\widetilde{\cE} \arrow{d} \arrow{r} & \cL' \arrow{d} \arrow{r} & \cL \arrow{d}\\
		\widetilde{B} \arrow{r}{\phi} \arrow[bend right, swap]{rr}{\widetilde{p}}& B' \arrow{r}{q} & S,
	\end{tikzcd}$$
	which satisfies the properties described in \Cref{reductionBetterCase}. We use this construction because we do not want to consider all the points of the fibers $p^{-1}(s)$ or $\widetilde{p}^{-1}(s)$: we need paths $\delta_{i}$ which leaves periods unchanged. Thus, we fix a point $b' \in q^{-1}(s)$ and consider only the points of $\phi^{-1}(b')$. Choosing only these points corresponds to fixing a branch of the square root $\sqrt{\lambda}$ and some other branch corresponding to the ramification at $\infty$; in fact, recall that the previous diagram is used to `separate the branch locus' of $\sigma$ and obtain the morphism $\phi$ which is unramified over the preimages of $0,\infty$. Unfortunately, our previous method does not remove the ramification over the whole of $\{0,1,\infty\}$: this means that into the fiber $\phi^{-1}(b')$ there are points which correspond to different choices of the square root $\sqrt{\lambda-1}$. If we consider two points $b_i, b_j \in \phi^{-1}(b')$ which correspond to the same branch of $\sqrt{\lambda}$ and $\sqrt{\lambda-r}$ but to different branches of $\sqrt{\lambda-1}$, then for each path $\delta$ connecting $b_i$ to $b_j$ the loop $\widetilde{p}(\delta)$ is not homotopically trivial in $S$. Therefore, in this situation our previous method fails to provide paths $\delta_i$ which satisfy \Cref{deltaij2}. In what follows we give a method to deal with some cases like this.
\end{exe}

\paragraph{Belyi pairs.}

Keep in mind the proof of \Cref{AbhyankarApproach} and \Cref{exampleAbhyankar}. The impossibility to remove the ramification over the whole of $\{0,1,\infty\}$ is due to the fact that we also want the compactification of $B'$ to be $\bP_1$ (we use this fact in the construction of $\delta_i$). We would be able to remove all three ramifications, but only allowing the genus of $B'$ at least $1$. Anyway, this attempt also fails: in fact, the growth of the genus of $B'$ gives rise to new generators of the fundamental group $\pi_1(\mathcal{B}'-A,b')$ which act non-trivially by monodromy on periods. This naturally leads to ask the following question: \emph{in which case the approach can be extended to deal with sections which are ramified into the whole of $\{0,1,\infty\}$?}

Here, the idea is to deal with the problem by using the theory of Belyi pairs and dessin d'enfants (see \cite{G} and \cite{JW} for details). By the term \emph{Belyi pair} we will refer to a pair $(\mathcal{B}, f)$ in which $\mathcal{B}$ is a compact Riemann surface and $f$ is a Belyi function, i.e. a morphism $f:\mathcal{B} \rightarrow \bP_1$ with exactly three branching values which we can assume to be $\{0,1,\infty\}$; two Belyi pairs will be considered equivalent when they are equivalent as ramified coverings. There exists a bijective map between equivalence classes of Belyi pairs and equivalence classes of dessins d'enfants: moreover the branching properties of the Belyi function are determined by the `graph theory' properties of the corresponding dessin d'enfant.

Now, look at the proof of \Cref{AbhyankarApproach}. In order to allow the idea of the proof to work we need a Belyi pair $(\bP_1,q)$ which play the role of the morphism $q:\mathcal{B}' \rightarrow \bP_1$ in \Cref{redBettCase}. Moreover, since we need the use of Abhyankar's Lemma, we want the Belyi function $q$ to be a Galois morphism. This corresponds to considering \emph{regular dessins}. We say that a regular dessin is of the type $(l,m,n)$ if the ramification indexes at preimages of $0,1,\infty$ are $l,m,n$ respectively. By an application of Riemann-Hurwitz formula it's possible to determine all the possibilities for the type $(l,m,n)$. First of all we obtain that in each case at least one of $l, m$ and $n$ must be equal to $2$; we may assume that $m = 2$. Simple arithmetic shows that the only solutions for the type $(l,m,n) = (l,2,n)$ are the following:
$(3,2,3), (3,2,4), (3,2,5)$ and $(2,2,n)$. These are all the cases where the proof of \Cref{AbhyankarApproach} can be extended to a section which ramifies in the whole of $\{0,1,\infty\}$.

\section{Further results and applications}

Here we end with some final considerations and examples. In particular, we will consider the case of Masser's section in details and we will determine generators of a full-rank subgroup of the monodromy group $M_\sigma^{\textnormal{rel}}$ in the general case. Thus, we obtain stronger effective results and leave some questions which could be interesting thinking to some possible applications.

\subsection{Some final questions}

We continue denoting by $\mathfrak{a}_0$ and $\mathfrak{a}_1$ two fixed and small enough circles centered at $0$ and $1$, respectively. We consider again a section which ramifies in the whole of $\{0,1,\infty\}$. Since the section is algebraic, then each branch point has finite branching order. Hence, by \Cref{periodFunctions} there exist loops $\zeta_0, \zeta_1$ in $B$ with base point $b_1$ such that
\begin{equation}\label{ramificationLoops}
	p(\zeta_0)=\mathfrak{a}_0^{z}, \quad  p(\zeta_1)=\mathfrak{a}_1^{z}, \qquad \textrm{for some } z \in \bZ.
\end{equation}
We want to point out that the proof of \Cref{AbhyankarApproach} works if we assume we have some control on the variation of $\log_\sigma$. To this end, suppose that $\log_\sigma$ has non-trivial variation along $\zeta_0$ or along $\zeta_1$; without loss of generality we put
$$
c_{\zeta_1}(\log_\sigma)=\log_\sigma(b_1) + \omega_{\zeta_{1}}, \qquad \omega_{\zeta_{1}}\neq0.
$$
Then we are assuming $\alpha=\zeta_1$ in \Cref{alpha}; from now on we write $\alpha$ in place of $\zeta$. The same proof as before continues to hold with some additional considerations: the branch locus $A$ of the morphism $\phi$ in \Cref{redBettCase} contains some preimage of $1$. Hence, the projections of paths $\delta_i$ defined in \Cref{delta_i2} are not homotopically trivial but they satisfy
$$
p(\delta_i)=\mathfrak{a}_1^{k_i}, \qquad \textrm{for some } k_i \in \bZ.
$$
By \Cref{periodFunctions} we get
\begin{equation}\label{periodComments}
	\omega(b_i):=c_{\delta_i}(\omega)=c_{\mathfrak{a}^{k_i}}(\omega), \qquad c_{\alpha_i}(\omega_j)=c_{\mathfrak{a}^z}(\omega_j^\cL)
\end{equation}
for each period function $\omega$ locally defined on $B$ and $j=1,2$. Analogously to what has been done in \Cref{control_omegai}, since $\omega_{\alpha_{1}}$ is non-zero there exists an index $i$ satisfying $\omega_{\alpha_i}\neq c_{\delta_i}(\omega_{\alpha_1})$. Let's say some words about this point: if we assume $\omega_{\alpha_i}= c_{\delta_i}(\omega_{\alpha_1})$ for all $i$'s, thanks to \Cref{ShiodaThm} we get $\sum_{i}c_{\delta_i}(\omega_{\alpha_1})=0$. Since the monodromy action of each $c_{\delta_{i}}$ is represented by a multiple of the triangular matrix associated to $\mathfrak{a}_{1}$, then the action of $\sum_{i}c_{\delta_i}$ is represented by an invertible matrix and the thesis follows. Let us fix an index $i$ which satisfies $\omega_{\alpha_i}\neq c_{\delta_i}(\omega_{\alpha_1})$. Let's consider the path $\delta_i$ and define the loop
$$
\Gamma:=\alpha_1\delta_i\alpha_i^{-1}\delta_i^{-1}.
$$
\Cref{figureGamma} aims at representing the loop $\Gamma$ when section ramifies at $1$.
\begin{figure}[h!]
	\centering
	\scalebox{0.6}{
		\begin{tikzpicture}[scale=0.6,xscale=0.8,>=latex, decoration={markings,mark=at position 0.64 with {\arrow{>}}}]
			
			\path (8,6) coordinate (b+);
			\path (8,12) coordinate (b-);
			\path (b+) to[out=180,in=-20] coordinate[pos=0.424](K1) coordinate[pos=0.56](K2) (1.7,8) to[out=160,in=155] (b-);
			
			\draw [color=blue!400,line width=2] {(b+) to[out=-30,in=190] (11,4.8) to[out=10,in=270] (12,5.5) to[out=90,in=5] (10.5,6.25) to[out=185,in=44] (9,5.7); 
			\draw [color=blue!400,line width=2,postaction={decorate}] (8.6,5.3) to[out=224,in=165] (9.5,4) to[out=-15,in=240] (13,5.5) to[out=60,in=44] (b+)};
			
			\draw [color=blue!400,line width=2] {(b-) to[out=-30,in=190] (11,10.8) to[out=10,in=270] (12,11.5) to[out=90,in=5] (10.5,12.25) to[out=185,in=44] (9,11.7); 
			\draw [color=blue!400,line width=2,postaction={decorate}] (8.6,11.3) to[out=224,in=165] (9.5,10) to[out=-15,in=240] (13,11.5) to[out=60,in=44] (b-)};
			
			\draw [color=blue!400,line width=2] (b+) to[out=195,in=-22] (K1);
			\draw [color=blue!400,line width=2] (K1) to[out=158,in=-20] (K2);
			\draw [color=blue!400,line width=2,postaction={decorate}] (K2)to[out=160,in=-45] (1.8,7.9) to[out=135,in=175] (b-);
			
			\filldraw[color=red] (b+) node[above left]{\Large $b_1$} ellipse (0.15 and 0.05);
			\filldraw[color=red] (b-) node[above left]{\Large $b_i$} ellipse (0.15 and 0.05);

			\node at(13.4,5.55)[above right]{\Large $\alpha_1$};
			\node at(13.4,11.55)[above right]{\Large $\alpha_i$};
			\node at(2.4,10.7)[above left]{\Large $\delta_i$}; 
			\node at(-2,8)[above left]{\Large $\bm{\Gamma}$};        
		\end{tikzpicture}
	}
	\caption{A visualization of the loop $\Gamma$.}\label{figureGamma}
\end{figure}
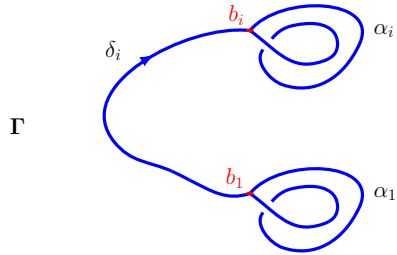

By \Cref{periodComments}, we get
$$
c_{\Gamma}(\omega)= \omega(b_1)
$$ 
for any period function $\omega$, i.e. periods don't change value in $b_1$ along $\Gamma$. With the same calculations as in \Cref{finalCalc} we obtain
$$
c_\Gamma(\log_\sigma)=\log_\sigma(b_1) + c_{\delta_i^{-1}}c_{\alpha_i^{-1}}(c_{\delta_i}(\omega_{\alpha_1}) - \omega_{\alpha_i}).
$$
Since the variation is non-trivial, we have obtained the thesis.

Anyway, we have no guarantees that the hypothesis for the variation of $\log_\sigma$ along $\zeta_0$ or $\zeta_1$ occur. Thus, the previous argument only apply to some particular cases. However it gives rise to some questions:
\begin{enumerate}
	\item Are we able to obtain some control on the variation of logarithms along $\zeta_0$ or/and $\zeta_1$?
	
	\item Are we able to find two loops of the type \Cref{Gamma} which give rise to independent variations of $\log_\sigma$?
\end{enumerate}

The first question aims at extending effectivity towards other ends, for example thinking to some applications in transcendence theory or finding explicit expressions of the monodromy action of $\pi_1(B,b)$ on determinations of the Betti map. The second question is a further push towards effectivity: thanks to our previous proofs we are now able to find an explicit loop which realizes the non-triviality of $M_\sigma^{\textnormal{rel}}$ but we also know that $M_\sigma^{\textnormal{rel}}\cong \bZ^2$. Answer the second question correspond to finding two loops which generate the full relative monodromy group (or a full-rank subgroup of it). In the final section we apply the techniques developed above, giving the concrete example of Masser's section and answering some of the previous questions.

\subsection{The case of Masser's section}

We precise that the following apply to all non-torsion algebraic sections with constant abscissa. For the sake of simplicity we choose one of them to give a specific example. Masser's section is an algebraic section of the Legendre scheme $\cL \rightarrow S$ defined by
$$
\sigma(\lambda)=\left(2,\sqrt{2(2-\lambda)}\right).
$$
In order to make it well defined we need the ramified cover $p: B \rightarrow S$ defined by taking the square root of $2-\lambda$. This cover has degree $2$ and is ramified above $\lambda=2$. We consider the elliptic scheme $\cE$ obtained as pullback of Legendre scheme via $p$: it can be seen as the elliptic curve over $\bC(\psi)$ defined by
$$
y^2=x(x-1)(x+\psi^2-2),
$$
where $\psi^2=2-\lambda$. The section $\sigma$ is pulled back to a well defined (rational) section $\widetilde{\sigma}(\psi)=(2,\sqrt{2}\psi)$, which is not identically torsion, since it is ramified over $S$. When we refer to $\log_\sigma$ (resp. $\sigma$) on $B$, we mean $\log_{\widetilde{\sigma}}$ (resp. $\widetilde{\sigma}$). Therefore, we have a diagram as in \Cref{diagramScheme}. Moreover, observe that a different choice of the square root would merely change sign to the section.

Our aim is finding a loop in $B$ where logarithm has a non-trivial variation but which leaves periods unchanged.

\paragraph{Choice of the right loop.}
	
	Fix a base point $s$ in $S-\{2\}$ and consider the loops $\mathfrak{a}_0$ and $\mathfrak{a}_1$ with base point $s$. Moreover, according to the notation used in \Cref{D0}, call $\mathfrak{d}$ a small loop in $S$ around $2$. The fiber $p^{-1}(s)$ contains two points, say $b_+, b_-$, which correspond to the two opposite square roots. Let $\delta$ be the lifting (via $p$) of $\mathfrak{d}$ with starting point $b_{+}$.
	
	Fix a branch $\log_\sigma$ of logarithm of $\sigma$ at $b_{+}$ and define $\log_{\sigma}(b_{-}):=c_{\delta}(\log_{\sigma})$. Since a different choice of the square root changes sign to $\sigma$, we have
	$$
	c_{\delta}(\sigma)=\sigma(b_-)=-\sigma(b_+).
	$$
	Hence $\log_\sigma(b_-)$ is a determination of $-\log_\sigma(b_+)$, so they differ by a period. In other words, we get
	\begin{equation}\label{omega_delta}
	c_{\delta}(\log_\sigma)=\log_\sigma(b_-)= -\log_\sigma(b_+)+\omega_\delta,
	\end{equation}
	where $\omega_\delta$ is a period. In order to simplify the calculations, we can replace the section $\sigma$ by a section $\widetilde{\sigma}$ so that we have
	$$
	\log_{\widetilde{\sigma}}=\log_\sigma-\frac{\omega_\delta}{2}.
	$$
	Observe that this transformation is made by adding a two-torsion section to $\sigma$ and since the two-torsion is well-defined over the base $B$ we don't need to extend the base. Moreover, note that adding a torsion section does not change our conclusions: in fact we are interesting in analytic continuations along loops which leaves unchanged periods. Let's continue denoting the new section $\widetilde{\sigma}$ by $\sigma$; thus we can assume $\omega_{\delta}=0$ in \Cref{omega_delta}, i.e.
	\begin{equation}\label{signLog}
	c_\delta(\log_\sigma)=-\log_{\sigma}(b_+).
	\end{equation}
	Now, let's consider $\mathfrak{a}_{0}$ and $\mathfrak{a}_{1}$. Put $\alpha_{+}, \alpha_{-}$ to be the liftings of $\mathfrak{a}_{1}$ with base points $b_{+},b_{-}$, respectively. By \Cref{wellDefLog} the variation of logarithm has to be non-trivial along one of the liftings of $\mathfrak{a}_{0}$ or $\mathfrak{a}_{1}$. Thus, without loss of generality we can assume
	\begin{equation}\label{logMasser}
		c_{\alpha_+}(\log_\sigma)=\log_\sigma(b_+) + \omega_+,
	\end{equation}
	where $\omega_+$ is a non-zero period. Let's consider the loop
	$$
	\gamma:=\mathfrak{a}_1\mathfrak{d}\mathfrak{a}_1^{-1}\mathfrak{d}^{-1}
	$$
	whose homotopy class in $\pi_1(S,s)$ is the identity. We define the loop $\Gamma$ as the lifting of $\gamma$ via $p$ with base point $b_+$, which is equivalent to putting
	\begin{equation}\label{GammaMasser}
	\Gamma=\alpha_+\delta\alpha_-^{-1}\delta^{-1}.
	\end{equation}
	Appealing to the notations of \Cref{commutators}, note that $\Gamma \in p_{*}^{-1}(D^{(1)})$.	The situation is clarified in \Cref{MasserSection}.
	\begin{figure}[h!]		\centering
		\scalebox{0.9}{
		\begin{tikzpicture}[scale=0.6,xscale=0.8,>=latex, decoration={markings,mark=at position 0.64 with {\arrow{>}}}]

			\path (-4,4) coordinate (A1);
			\path (12,4) coordinate (B1);
			\path (13,7) coordinate (C1);
			\path (-2,7) coordinate (D1);
			\path (-4,8) coordinate (A2);
			\path (12,8) coordinate (B2);
			\path (13,11) coordinate (C2);
			\path (-2,11) coordinate (D2);
			\path (-3,5.5) coordinate (M1);
			\path (12.5,5.5) coordinate (N1);
			\path (-3,9.5) coordinate (M2);
			\path (12.5,9.5) coordinate (N2);
			\path (M2) to[out=0,in=180] coordinate[pos=0.5](O) (N1);
			\path[name path=border5] (A1) to[out=0,in=180] coordinate[pos=0.5](I) (B2);
			
			\draw[line width=2] (C1) to[out=180,in=-15] coordinate[pos=0.767](K3) (O);
			\draw (O) to[out=195,in=0] coordinate[pos=0.49](K4) (D1);
			\path (8,6) coordinate (b+);
			\path (8,10) coordinate (b-);
						
			\shade[left color=gray!35, right color=gray!30]
			(I) to[out=-10,in=180] (N1) -- (C1) to[out=180,in=-15] (O) -- cycle;
			
			\shade[left color=gray!50, right color=gray!45] 
			(M1) to[out=0,in=200] (I) -- (O) to[out=195,in=0] (D1) -- cycle;
			
			\shade[left color=gray!50, right color=gray!34] 
			(A1) to[out=0,in=180] (B2) -- (N2) to[out=180,in=0] (M1) -- cycle;
			
			\path (O) to[out=200,in=0] coordinate[pos=0.7](K6) (M1);
			\draw[gray!100!white, very thick , path fading=south,postaction={draw, gray!47!white, path fading=north},,line width=2] (O) to[out=200,in=10] (K6);
			
			\shade[left color=gray!50, right color=gray!30] 
			(A2) to[out=0,in=180] (B1) -- (N1) to[out=180,in=0] (M2) -- cycle;
			\path (N1) to[out=180,in=0] coordinate[pos=0.28](K7) coordinate[pos=0.7](K8) (M2);
			\draw[gray!100!white, very thick , path fading=south,postaction={draw, gray!35!white, path fading=north},,line width=2] (O) to[out=-30,in=160] (K7);
			
			\shade[left color=gray!40, right color=gray!30] 
			(I) to[out=22,in=180] (B2) -- (N2) to[out=180,in=22] (O) -- cycle;
			
			\shade[left color=gray!50, right color=gray!30]
			(M2) to[out=0,in=165] (O) to[out=20,in=180] (N2) -- (C2) -- (D2) -- cycle;
			
			\path (N2) to[out=180,in=0] coordinate[pos=0.3](K5) (M1);
			\draw[gray!37!white, very thick , path fading=south,postaction={draw, gray!50!white, path fading=north},line width=2] (O) to[out=40,in=190] (K5);
			
			\draw[gray!43!white, very thick , path fading=south,postaction={draw, gray!45!white, path fading=north},line width=2] (O) to[out=125,in=-10] (K8);
			
			\path[draw,name path=border1] (A2) to[out=0,in=180] (B1);
			\draw (B1) to (C1);
			\draw (D1) to (A1);
			\draw (A1) to[out=0,in=180] (B2);
			\draw (B2) to (C2);
			\draw (C2) to (D2);
			\draw (D2) to (A2);
			
			\shade[left color=gray!15,right color=gray!45] (O) to[out=110,in=180] (5.7,8.5) to[out=0,in=70] (O);
			
			
			\path (-4,0) coordinate (A);
			\path (12,0) coordinate (B);
			\path (13,3) coordinate (C);
			\path (-2,3) coordinate (D);

			\shade[left color=gray!50, right color=gray!30]
			(A) -- (B) -- (C) -- (D) -- cycle;
			
			\draw[line width=2] (A) -- (B) -- (C) -- (D) -- cycle;
			
			\path (5,1.7) coordinate (2);
			\filldraw[color=red] (2) node[below left]{\Large $2$} ellipse (0.15 and 0.05);
			
			\path (8,2) coordinate (s);
			
			\path (10,1.5) coordinate (1);
			\filldraw[color=white] (1) ellipse (0.15 and 0.05);
			\draw (1) node[below right]{$1$} ellipse (0.15 and 0.05);
			
			\path (0,1.5) coordinate (0);
			\filldraw[color=white] (0) ellipse (0.15 and 0.05);
			\draw (0) node[above right]{$0$} ellipse (0.15 and 0.05);
			
			\draw [color=blue!400,line width=2,postaction={decorate}] (s) to[out=-30,in=190] (11.5,0.8) to[out=10,in=5] (s);
			\draw [color=blue!400,line width=2,postaction={decorate}] (s) to[out=150,in=160] (2,0.7) to[out=-20,in=-150] (s);
			\filldraw[color=red] (s) node[above right]{\Large $s$} ellipse (0.15 and 0.05);
			
			
			\path (10,5.5) coordinate (1+);
			\path (10,9.5) coordinate (1-);
			\filldraw[color=white] (1+) ellipse (0.15 and 0.05);
			\filldraw[color=white] (1-) ellipse (0.15 and 0.05);
			\draw (1+) ellipse (0.15 and 0.05);
			\draw (1-) ellipse (0.15 and 0.05);
			
			\path (0,5.2) coordinate (0+);
			\path (0,9.2) coordinate (0-);
			\filldraw[color=white] (0+) ellipse (0.15 and 0.05);
			\draw (0+) ellipse (0.15 and 0.05);
			\filldraw[color=white] (0-) ellipse (0.15 and 0.05);
			\draw (0-) ellipse (0.15 and 0.05);
			
			\path (b+) to[out=180,in=-20] coordinate[pos=0.424](K1) coordinate[pos=0.56](K2) (1.7,8) to[out=160,in=155] (b-);
			
			\draw [color=blue!400,line width=2,postaction={decorate}] {(b+) to[out=-30,in=190] (11.5,4.8) to[out=10,in=5] (b+)};
			\draw [color=blue!400,line width=2,postaction={decorate}] (b-) to[out=-30,in=190] (11.5,8.8) to[out=10,in=5] (b-);
			\draw [color=blue!400,line width=2] (b+) to[out=195,in=-22] (K1);
			\draw [color=blue!40,line width=2,dotted] (K1) to[out=158,in=-20] (K2);
			\draw [color=blue!400,line width=2,postaction={decorate}] (K2)to[out=160,in=-45] (1.8,7.9) to[out=135,in=155] (b-);
			\filldraw[color=red] (b+) node[above right]{\Large $b+$} ellipse (0.15 and 0.05);
			\filldraw[color=red] (b-) node[above right]{\Large $b-$} ellipse (0.15 and 0.05);
			\draw[gray!20!white, very thick, path fading=south,postaction={draw, gray!35!white, path fading=north}, line width=2] (I) -- (O);
			\filldraw[color=red] (O) node[below right]{\Large $O$} ellipse (0.15 and 0.05);
			
			\draw[line width=2] (K4) to[out=184,in=0] (D1) -- (A1) to[out=0,in=180] (B2) -- (C2) -- (D2) -- (A2) to[out=0,in=180] (B1) -- (C1) to[out=180,in=-4.8] (K3);
			
			\node at(11,1.55)[above right]{\Large $\mathfrak{a}_1$};
			\node at(2,2)[above right]{\Large $\mathfrak{d}$};
			\node at(11,5.55)[above right]{\Large $\alpha_+$};
			\node at(11,9.55)[above right]{\Large $\alpha_-$};
			\node at(2,9.7)[above right]{\Large $\delta$};
			\node at(-6,1.4){\scalebox{3}{$\textbf{\textit{S}}$}};
			\node at(-6,9.7){\scalebox{3}{$\textbf{\textit{B}}$}};
		\end{tikzpicture}
	}
		\caption{A visualization of the covering map $B\rightarrow S$ and the relevant loops.}\label{MasserSection}
	\end{figure}
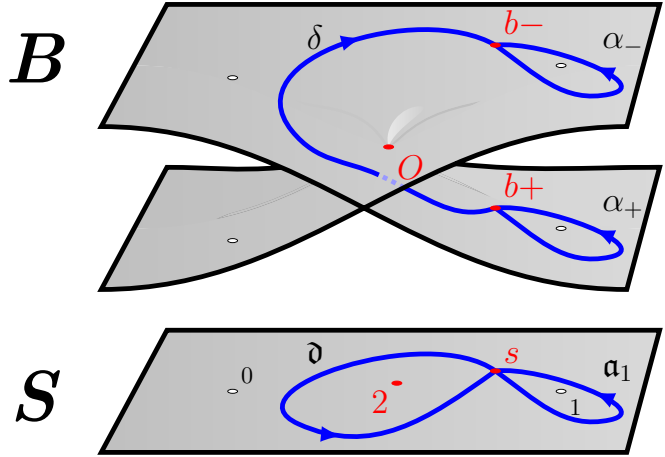
	
\paragraph{Monodromy of periods and logarithm along the loop.}
	
	Let us study the monodromy of periods along $\Gamma$. As usual, we define a basis of periods by the equations
	$$
	\omega_1=\omega_1^{\mathcal{L}}\circ p, \qquad \omega_2=\omega_2^{\mathcal{L}}\circ p.
	$$
	Hence, we get
	$$
	c_{\delta}(\omega_i)= c_{\mathfrak{d}}(\omega_i^{\mathcal{L}})=\omega_i(b_+) \qquad \textrm{for }i=1,2.
	$$
	In other words, this means that $\omega_i(b_+)=\omega_i(b_-)$. Analogously, we obtain
	$$
	c_{\alpha_+}(\omega_i)= c_{\alpha_-}(\omega_i) = c_{\mathfrak{a}_1}(\omega_i^{\mathcal{L}}).
	$$
	Therefore, for $i=1,2$ we have
	$$
	c_{\Gamma}(\omega_i)=c_{\delta^{-1}}c_{\alpha_-^{-1}}c_{\delta}c_{\alpha_+}(\omega_i)= c_{\alpha_+^{-1}} c_{\alpha_+}(\omega_i)= \omega_i(b_+).
	$$
	In other words, periods have no monodromy along $\Gamma$. Now, we want to study the analytic continuation of logarithm along $\Gamma$. Let us consider the loops $\alpha_+, \alpha_-:I \rightarrow B$. Since they are liftings of the same loop $\mathfrak{a}_{1}$ and the section changes sign when changing branch of the square root, we get
	$$
	\sigma(\alpha_+(t)) + \sigma(\alpha_-(t))=0 \quad \textrm{for each } t \in I.
	$$
	In terms of logarithms, this means that
	$$
	\log_{\sigma}(\alpha_+(t)) + \log_{\sigma}(\alpha_-(t))= \omega(t),
	$$
	where $\omega$ is a period function. By \Cref{signLog}, we get $\omega(0)=0$. Since $\omega(t)$ is an integral combination of a basis of periods, it follows that $\omega(t)= 0$ for each $t\in I$. This implies 
	$$
	\log_{\sigma}(\alpha_-(t))=-\log_{\sigma}(\alpha_+(t)) \quad \textrm{for each } t \in I.
	$$
	In particular this means that the analytic continuation of $\log_{\sigma}$ along $\alpha_-$ is the same as the analytic continuation of $-\log_{\sigma}$ along $\alpha_+$, i.e.
	$$
	c_{\alpha_-^{-1}}(\log_{\sigma}) = -c_{\alpha_+^{-1}}(\log_{\sigma}).
	$$
	Moreover, by \Cref{logMasser} and \Cref{lemmaMonodromy} we get
	$$
	c_{\alpha_+^{-1}}(\log_\sigma)  = \log_\sigma(b_+) -c_{\alpha_+^{-1}}(\omega_+).
	$$
	Thus, appealing to all previous considerations we obtain
	\begin{align*}
		c_{\delta}c_{\alpha_+}(\log_{\sigma})&= c_{\delta}(\log_{\sigma}+\omega_+) = c_{\delta}(\log_{\sigma}) + c_{\delta}(\omega_+)=\\
		& = \log_{\sigma}(b_-) + \omega_+(b_+).
	\end{align*}
	Using again the previous relations we get
	\begin{align*}
		c_{\alpha_-^{-1}}c_{\delta}c_{\alpha_+}(\log_{\sigma})&= c_{\alpha_-^{-1}}(\log_{\sigma}) + c_{\alpha_-^{-1}}(\omega_+)=\\
		&= -c_{\alpha_+^{-1}}(\log_{\sigma}) + c_{\alpha_+^{-1}}(\omega_+) =\\
		&= -\log_{\sigma}(b_+) + 2c_{\alpha_+^{-1}}(\omega_+)=\\
		&= \log_{\sigma}(b_-) + 2c_{\alpha_+^{-1}}(\omega_+).
	\end{align*}
	Finally we have
	\begin{equation}\label{variationMasser}
	\begin{aligned}
		c_{\Gamma}(\log_{\sigma})&=c_{\delta^{-1}}c_{\alpha_-^{-1}}c_{\delta}c_{\alpha_+}(\log_{\sigma}) = c_{\delta^{-1}}(\log_{\sigma})+ 2c_{\alpha_+^{-1}}(\omega_+)=\\
		&= \log_{\sigma}(b_+) + 2c_{\alpha_+^{-1}}(\omega_+).
	\end{aligned}
	\end{equation}
	Therefore, the value of the logarithm in $b_+$ has changed by a period $2c_{\alpha_+^{-1}}(\omega_+)$ which is non-zero by \Cref{logMasser}. The loop $\Gamma$ provides the non-triviality of the group $M_{\sigma}^{\textnormal{rel}}$.
	
	To be more precise, we have just proved the existence of a loop $\Gamma_1:=\Gamma$ which generates a non-trivial variation of logarithm of Masser's section leaving periods of the corresponding scheme unchanged. Trying to find explicit generators of the full relative monodromy group, analogously to \Cref{GammaMasser} we can define a loop $\Gamma_{0}$ by replacing $\alpha_{+}$ and $\alpha_{-}$ with liftings of $\mathfrak{a}_{0}$, which we denote by $\beta_{+}, \beta_{-}$ for a moment. Unfortunately, we are not able to prove that the monodromy action of $\Gamma_{0}$ on $\log_{\sigma}$ is non-trivial. Anyway, we are able to find another loop $\Gamma'$ depending on $\Gamma$ such that the monodromy actions of $\Gamma$ and $\Gamma'$ generate a full-rank subgroup of $M_{\sigma}^{\textnormal{rel}}$; we will do this in full generality in the final part of the paper. In the specific case of Masser's section, the loop $\Gamma'$ will be of the form $\Gamma'=\alpha \Gamma \alpha^{-1}\Gamma^{-1}$ where $\alpha \in \{\alpha_{+}, \beta_{+}\}$; keeping the notations of \Cref{commutators}, we obtain $\Gamma \in p_{*}^{-1}\left(D^{(1)}\right)$ and $\Gamma' \in p_{*}^{-1}\left(D^{(2)}\right)$ in this case.

\subsection{Generating a full-rank subgroup of $M_{\sigma}^{\textnormal{rel}}$ with effectivity}

Put ourselves in the same setting as in \Cref{effectiveProofSection} and keep the same notations. Moreover, let $z \in \bZ$ be the least common multiple of ramification indexes of points of $B$ which lie over $0$ or $1$ via the map $p:B \rightarrow S$. Denote by $\zeta_{0}$ and $\zeta_{1}$ the loops in $B$ with base point $b_{1}$ satisfying \Cref{ramificationLoops}. Here we are going to prove a version of \Cref{mainTheoremTrue}, which is stronger than before in the sense of effectivity: in fact, we are going to exhibit explicit generators of a full-rank subgroup of $M_{\sigma}^{\textnormal{rel}}$.

\begin{thm}\label{mainTheoremStrong}
	Let $\sigma:B \rightarrow \cE$ be a non-torsion section of the elliptic scheme and let $\Gamma$ be the loop constructed in \Cref{mainTheoremTrue}. There exists a loop $\Gamma'$ in $B$ depending on $\Gamma$ and $z$ such that the monodromy actions of $\Gamma$ and $\Gamma'$ generate a subgroup of $M_{\sigma}^{\textnormal{rel}}$ which is isomorphic to $\bZ^{2}$; actually, $\Gamma'$ is of the form
	$$
	\Gamma'=\zeta\Gamma\zeta^{-1}\Gamma^{-1},
	$$
	where $\zeta \in \{\zeta_{0},\zeta_{1}\}$; further, keeping the notations of \Cref{commutators} we have $\Gamma' \in p_{*}^{-1}\left(D^{(z+1)}\right)$. In particular, the relative monodromy group of logarithm of $\sigma$ with respect to periods of $\cE \rightarrow B$ is isomorphic to $\bZ^2$.
\end{thm}

\begin{proof}
	Thanks to \Cref{mainTheoremTrue} we know that $\Gamma$ induces a trivial monodromy action on periods and a non-trivial variation on $\log_{\sigma}$. We write 
	$$
	c_{\Gamma}(\log_{\sigma})=\log_{\sigma}(b_{1})+\omega_{\Gamma},
	$$
	where $\omega_{\Gamma}:=n_{1}\omega_{1} + n_{2}\omega_{2}$ is a non-zero period and $n_{1},n_{2}$ are integers not both zero. Without loss of generality, let us assume $n_{1}\neq 0$ and consider the loop $\zeta_{1}$. (Otherwise, if we are in the case $n_{1}= 0$ we must have $n_{2}\neq 0$ and we can proceed in the same way but considering $\zeta_{0}$ in place of $\zeta_{1}$.) Define
	$$
	\Gamma':=\zeta_{1}\Gamma\zeta_{1}^{-1}\Gamma^{-1}
	$$
	and notice that $\Gamma'$ induces trivial monodromy action on periods. By analytic continuation we get
	$$
	c_{\Gamma'}(\log_{\sigma})=\log_{\sigma}(b_{1}) + c_{\zeta_{1}^{-1}}(\omega_{\Gamma}) - \omega_{\Gamma},
	$$
	where $c_{\zeta_{1}^{-1}}(\omega_{\Gamma}) - \omega_{\Gamma}=-2n_{1}z\omega_{2}$. Therefore, $\Gamma$ and $\Gamma'$ give rise to independent variations of $\log_{\sigma}$, so that their monodromy actions generate a subgroup of $M_{\sigma}^{\textnormal{rel}}$ which is isomorphic to $\bZ^{2}$. Finally, since $\Gamma \in p_{*}^{-1}\left(D^{(1)}\right)$ and by \Cref{ramificationLoops}, we get $\Gamma' \in p_{*}^{-1}\left(D^{(z+1)}\right)$.\\
\end{proof}

\bibliographystyle{hplain}

\bibliography{biblio}

\Addresses

 \end{document}